\documentclass[12pt,leqno]{amsart}
\usepackage{amsmath, amssymb, amscd, amsfonts,    stmaryrd, turnstile, mathrsfs, eucal}

\newtheorem{theorem}{Theorem}[section]
\newtheorem{lemma}[theorem]{Lemma}
\newtheorem{proposition}[theorem]{Proposition}
\newtheorem{corollary}[theorem]{Corollary}

\theoremstyle{definition}
\newtheorem{definition}[theorem]{Definition}
\newtheorem{example}[theorem]{Example}
\newtheorem{remark}[theorem]{Remark}

\begin{document}

\title[Closures and co-closures]{Closures and co-closures attached to FCP ring extensions}

\author[G. Picavet and M. Picavet]{Gabriel Picavet and Martine Picavet-L'Hermitte}
\address{Math\'ematiques \\
8 Rue du Forez, 63670 - Le Cendre\\
 France}
\email{picavet.mathu (at) orange.fr}

\begin{abstract} The paper deals with ring extensions $R\subseteq S$ and the poset $[R,S]$ of their subextensions, with a special look at FCP extensions (extensions such that $[R,S]$ is Artinian and Noetherian).  When the extension has FCP, we show that there exists a co-integral closure, that is a least element $\underline R$ in $[R,S]$ such that $\underline R \subseteq S$ is integral. Replacing the integral property by the integrally closed property, we are able to prove a similar result for an FCP extension. 
 The radicial closure of $R$ in $S$ is well known. We are able to exhibit a suitable separable closure of $R$ in $S$ in case the extension has FCP, and then results are similar to those of field theory. The FCP property being always guaranteed, we discuss when an extension has a co-subintegral or a co-infra-integral closure. Our theory is made easier by using anodal extensions. These (co)-closures exist for example when the extension is catenarian, an interesting special case for the study of distributive extensions to appear in a forthcoming paper. 
\end{abstract} 



\maketitle

\section{Introduction and Notation}

We consider the category of commutative and unital rings, whose    epimorphisms will be involved. If $R\subseteq S$ is a (ring) extension, we denote by $[R,S]$ the set of all $R$-subalgebras of $S$. We set $]R,S[: =[R,S]\setminus\{R,S\}$ (with a similar definition for $[R,S[$ or $]R,S]$). 
Then, $[R,S]$ endowed with the partial order $\subseteq$ is a lattice,  called  the lattice of the extension. 

Now, $(R:S)$ is the conductor of the extension $R\subseteq S$. The integral closure of $R$ in $S$ is denoted by $\overline R^S$ (or by $\overline R$ if no confusion can occur).

The aim of this paper is to introduce or revisit various relative closures or  hulls associated to ring extensions. But we will mainly work in the context of FCP or FIP extensions, whose meanings are defined below. We emphasize that these restricted contexts produce results that may not be valid in an arbitrary context. Actually, some of them are established because we need them in forthcoming papers, for example when we characterize extensions, whose lattices  are distributive. Note also that many of them have their own interest.

The extension $R\subseteq S$ is said to have FIP (for the ``finitely many intermediate algebras property") or is an FIP  extension if $[R,S]$ is finite. A {\it chain} of $R$-subalgebras of $S$ is a set of elements of $[R,S]$ that are pairwise comparable with respect to inclusion.   
 We will say that $R\subseteq S$ is {\it chained} if $[R,S]$ is a chain. We  also say that the extension $R\subseteq S$ has FCP (or is an FCP extension) if each chain in $[R,S]$ is finite, or equivalently, its lattice  is Artinian and Noetherian. Clearly, each extension that satisfies FIP must also satisfy FCP. Dobbs and the authors characterized FCP and FIP extensions \cite{DPP2}. 

Our main tool are the minimal (ring) extensions, a concept that was introduced by Ferrand-Olivier \cite{FO}. Recall that an extension $R\subset S$ is called {\it minimal} if $[R, S]=\{R,S\}$. An extension $R\subseteq S$ is called {\it simple} if $S=R[t]$ for some $t\in S$.     
A minimal extension is simple. The key connection between the above ideas is that if $R\subseteq S$ has FCP, then any maximal (necessarily finite) chain $\mathcal C$ of $R$-subalgebras of $S$, $R=R_0\subset R_1\subset\cdots\subset R_{n-1}\subset R_n=S$, with {\it length} $\ell(\mathcal C):=n <\infty$, results from juxtaposing $n$ minimal extensions $R_i\subset R_{i+1},\ 0\leq i\leq n-1$. An FCP extension is finitely generated, and (module) finite if integral. For any extension $R\subseteq S$, the {\it length} $\ell[R,S]$ of $[R,S]$ is the supremum of the lengths of chains of $R$-subalgebras of $S$. Notice that if $R\subseteq S$ has FCP, then there {\it does} exist some maximal chain of $R$-subalgebras of $S$ with length $\ell[R,S]$ \cite[Theorem 4.11]{DPP3}. 
 An FCP extension $R\subset S$ is called {\it catenarian} if all the maximal chains of $[R,S]$ have the same length. We study such extensions in \cite{Pic 12}.
 
Recall that an extension $R\subseteq S$ is called {\it Pr\"ufer} if $R\subseteq T$ is a flat epimorphism for each $T\in[R,S]$ (or equivalently, if $R\subseteq S$ is a normal pair) \cite[Theorem 5.2, p. 47]{KZ}. It follows that a Pr\"ufer integral extension is an isomorphism. The {\it Pr\"ufer hull} of an extension $R\subseteq S$ is the greatest {\it Pr\"ufer} subextension $\widetilde R$ of $[R,S]$ \cite{Pic 3}. 
  In \cite{Pic 5}, we defined an extension $R\subseteq S$ to be {\it quasi-Pr\"ufer} if it can be factored $R\subseteq R'\subseteq S$, where $R\subseteq R'$ is integral and $R'\subseteq S$ is Pr\"ufer. An FCP extension is quasi-Pr\"ufer \cite[Corollary 3.4]{Pic 5}. 
     
 An extension $R\subseteq S$ is called {\it almost-Pr\"ufer} if $\widetilde R\subseteq S$ is integral, or equivalently, when $R\subseteq S$ is FIP, if $S=\widetilde R\overline R$ \cite[Theorem 4.6]{Pic 5}. An almost-Pr\"ufer extension  is quasi-Pr\"ufer.

\subsection{A summary of the main results} Any undefined material
 is explained in the next sections and at the end of this section, where the reader may find standard results on minimal extensions. 

Ring extensions admit numerous closures and we add in this paper some more. One of our aims is to define dual closures, for example like the co-integral closure of a ring extension.  These dual closures exist for FCP extensions but the reader is warned that they do not for arbitrary extensions. We think that the main reason is that FCP extensions have Artinian and Noetherian lattices although   proofs   most of time rely on induction over tower of minimal extensions. 

For a ring extension $R\subseteq S$, Section 2 recalls some known facts about the seminormalization and the $t$-closure of $R$ in $S$. We also recall less known facts about the $u$-closure of $R$ in $S$, that can be found in a comprehensive study of the first author \cite{Pic 0}. Also some technical results about these closures, that are needed in the sequel, are established in the FCP context. We have postponed deeper results on the $u$-closure in Section 5, because it is interwoven with another closure. Theorem \ref{1.313} shows that the t-closure of an integral FCP extension is the composite of its u-closure and its seminormalization. 
  
We show that an FCP extension $R\subseteq S$ has a co-integral closure $\underline R$ containing $\widetilde R$, with equality if and only if the extension is almost-Pr\"ufer (Theorems \ref{0.19} and \ref{0.21}). We give an FIP example of an almost-Pr\"ufer extension (Example \ref{0.23}). 

In Section 3, we consider the elements $T\in[R,S]$ such that $R\subseteq T$ is integrally closed. Evidently the Pr\"ufer hull of $R$ in $S$ is such an element. We are interested in extensions that have a greatest integrally closed subextension, for example if the extension is chained. Another example is a quasi-Pr\"ufer extension (for example, an FCP extension). In some cases, this greatest element $T$ verifies $T\subseteq S$ is integral, and could have been considered as a co-integral closure, but it is not always the case. Actually the right co-integral closure is defined in Section 3.

We consider in Section 4 some closures linked to the properties of their residual field extensions. Kubota introduced the radicial closure of a ring extension \cite{KU}, reconsidered by Manaresi under another name \cite{MA}. The second author considered a weaker closure \cite{Pic 14}.  Another type of ring extensions, that are linked to residual extensions, are the unramified extensions of Algebraic Geometry. In order to unify these concepts, if $\mathcal P$ is a property of field extensions, we say that an extension is a $\kappa$-$\mathcal P$ extension if all its residual extensions verify $\mathcal P$. For example a radicial extension is a $\kappa$-radicial $i$-extension. As a first result we show that an unramified radicial FCP integral extension is trivial (Proposition \ref{SEP}). We give some properties of unramified extensions in the FCP context. We introduce a new concept: the $\kappa$-separable extensions, within the framework of integral extensions. This concept is more tractable than the unramified property. They coincide in some cases: an integral $t$-closed FCP extension is unramified if and only if it is $\kappa$-separable (Corollary \ref{RAM1}). An important result is that an FCP integral extension $R\subseteq S$ admits a $\kappa$-separable closure ${}^\square_SR$ such that ${}^\square_SR \subseteq S$ is radicial (Theorem \ref{1.20} and Corollary \ref{1.21}). 

In Section 5, we give properties of $u$-integral extensions and anodal ($u$-closed)  extensions, valid in the FCP context. For example, an FCP integral extension is $u$-integral if and only if it is locally an epimorphism, in which case the extension is unramified (Theorem \ref{1.320}). Moreover, an integral FCP extension is anodal if and only if it is an $i$-extension (Proposition \ref{1.311}). We also examine the links between the $u$-closure and the other closures or properties previously considered. We then introduce the unramified closure ${}^\omega_SR$ of a ring extension $R\subseteq S$ and show that ${}^\omega_SR \subseteq S$ is $u$-closed in case the extension is integral and has FCP (Proposition  \ref{1.14}). Results of Section 5  about the radicial property may not hold. 

In Section 6, we consider dual closures of the seminormalizations and $t$-closures of  ring extensions. This may seem weird, but in a forthcoming paper on distributive FCP extensions, these new closures are an essential tool. The reader is warned that these closures may not exist but exist in the context we have to consider. A co-subintegral closure of a ring extension $R\subseteq S$ is the least $T\in[R,S]$ such that $T\subseteq S$ is subintegral. The co-infra-integral closure is similarly defined. We give some generalities about these closures and their properties. For example,  an infra-integral FCP extension admits a co-subintegral closure. We then compare all the closures we have introduced with these dual closures. Now considering an integral FCP extension $R\subseteq S$ and setting $T=\cap[U\in[R,S]\mid U\subseteq S$ subintegral$]$, we get that the co-subintegral closure exists if and only if $T\subseteq S$ is catenarian (resp.; subintegral), in which case $T$ is the co-subintegral closure (Theorem \ref{1.425}). As a corollary, we have that if in addition the extension is catenarian, the co-subintegral closure exists. We have similar results for the co-infra-integral closure. But the proofs are much more involved. The rest of the paper exhibits examples where these dual closures   exist or not.

Section 7 explores the cardinality of lattices of integral FIP extensions.
Some results are linked to co-subintegral closures and complements. 
 A striking result is Theorem \ref{B5.8} which says that if  $R\subset S$ is an FIP integral  extension over the local ring $R$, then  $|[R,S]|$ is the sum of the number of the complements of the t-closure ${}_S^tR$ in all extensions $R'\subseteq S'$, for each $(R',S')\in[R,{}_S^tR]\times[{}_S^tR,S]$.
   
\subsection{Definitions   and notation} 
A {\it local} ring is here what is called elsewhere a quasi-local ring. As usual, Spec$(R)$ and Max$(R)$ are the set of prime and maximal ideals of a ring $R$. 
 For an extension $R\subseteq S$ and an ideal $I$ of $R$, we write  
  $\mathrm{V}_S(I):=\{P\in\mathrm{Spec}(S)\mid I\subseteq P\}$ and $\mathrm D_ S(I)$ for its complement.
  The support of an $R$-module $E$ is $\mathrm{Supp}_R(E):=\{P\in\mathrm{Spec}(R)\mid E_P\neq 0\}$, and $\mathrm{MSupp}_R(E):=\mathrm{Supp}_R(E)\cap\mathrm{Max}(R)$. 
  We denote by ${\mathrm L}_R(E)$ the length of $E$ as an $R$-module.
  
When $R\subseteq S$ is an extension, we will set $\mathrm{Supp}(T/R):=\mathrm{Supp}_R(T/R)$ and $\mathrm{Supp}(S/T):=\mathrm{Supp}_R(S/T)$ for each $T\in [R,S]$, unless otherwise specified. 
  
Let $R\subset S$ be a ring extension. For $T\in[R,S]$, an element $U\in[R,S]$ is called a {\it complement} of $T$ if $T\cap U=R$ and $TU=S$. If $T$ has a unique complement, this complement is denoted by $T^o$. 
 
If $R\subset S$ is an extension and $T\in]R,S[$, we say that the extension {\it splits} at $T$, if $\mathrm{MSupp}(S/T)\cap\mathrm{MSupp}(T/R)=\emptyset$. In \cite[Theorem 4.8]{Pic 15}, we proved that if an FCP extension splits at $T$, then $T$ has a unique complement $T^o$.
 
Let $R\subset S$ be an FCP extension with 
 $\mathrm{MSupp}(S/R):=\{M_i\}_{i=1}^n$. In \cite[the paragraph before Proposition 2.21]{Pic 10}, we say that $R\subset S$ is a {\it$\mathcal B$-extension} if the map $\varphi:[R,S]\to\prod_{i=1}^n[R_{M_i},S_{M_i}]$ defined by $\varphi(T):=(T_{M_i})_{i=1}^n$ is bijective. In this case, $\varphi$ is an order isomorphism. An integral FCP extension is a $\mathcal B$-extension \cite[Theorem 3.6]{DPP2}. If, in addition, $R\subset S$ has FIP, then $|[R,S]|=\prod_{i=1}^n|[R_{M_i},S_{M_i}]|$.
  
If $R\subseteq S$ is a ring extension and $P\in\mathrm{Spec}(R)$, then $S_P$ is both the localization $S_{R\setminus P}$ as a ring and the localization at $P$ of the $R$-module $S$. We denote by $\kappa_R(P)$ the residual field $R_P/PR_P$ at $P$. If $R\subseteq S$ is a ring extension and $Q\in\mathrm{Spec}(S)$, there exists a residual field extension $\kappa_R(Q\cap R)\to\kappa_S(Q)$. Moreover, for $P\in \mathrm{Spec}(R)$, the ring morphism $\kappa(P)\to\kappa(P)\otimes_RS=S_P/PS_P$ is called the fiber morphism at $P$, whose target is the fiber of  the extension at $P$.

A ring extension $R\subseteq S$ is called an {\it i-extension} if the natural map $\mathrm{Spec}(S)\to\mathrm{Spec}(R)$ is injective. 

Finally, $|X|$ is the cardinality of a set $X$, $\subset$ denotes proper inclusion and for a positive integer $n$, we set $\mathbb{N}_n:=\{1,\ldots,n\}$. The characteristic of an integral domain $k$ is denoted by $\mathrm{c}(k)$. 

\subsection{Results on minimal extensions}
 
\begin{definition}\label{crucial 1}\cite[Definition 2.1]{Pic 15}  An extension $R\subset S$ is called {\it $M$-crucial} if $\mathrm{MSupp}(S/R)=\{M\}$. Such a maximal ideal $M$ is called the {\it crucial (maximal) ideal} $\mathcal{C}(R,S)$ of $R\subset S$. 
 \end{definition}

\begin{theorem}\label{crucial}\cite[Th\'eor\`eme 2.2]{FO} A minimal extension is simple, crucial and is either integral (finite) or a flat epimorphism.
\end{theorem} 

In \cite{Pic 5}, a minimal flat epimorphism is called a {\it Pr\"ufer minimal} extension. An FCP Pr\"ufer extension has FIP and is a tower of finitely many Pr\"ufer minimal extensions \cite[Proposition 1.3]{Pic 5}.

Three types of minimal integral extensions exist, characterized in the next theorem, (a consequence of the fundamental lemma of Ferrand-Olivier), whence four types of minimal extensions exists and are mutually exclusive.

\begin{theorem}\label{minimal} \cite [Theorems 2.2 and 2.3]{DPP2} Let $R\subset T$ be an extension and  $M:=(R: T)$. Then $R\subset T$ is minimal and finite if and only if $M\in\mathrm{Max}(R)$ and one of the following three conditions holds:

\noindent (a) {\bf inert case}: $M\in\mathrm{Max}(T)$ and $R/M\to T/M$ is a minimal field extension.

\noindent (b) {\bf decomposed case}: There exist $M_1,M_2\in\mathrm{Max}(T)$ such that $M=M _1\cap M_2$ and the natural maps $R/M\to T/M_1$ and $R/M\to T/M_2$ are both isomorphisms; or, equivalently, there exists $q\in T\setminus R$ such that $T=R[q]$ and $q^2-q\in M$.

\noindent (c) {\bf ramified case}: There exists $M'\in\mathrm{Max}(T)$ such that ${M'}^2\subseteq M\subset M',\  [T/M:R/M]=2$, and the natural map $R/M\to T/M'$ is an isomorphism;
 or, equivalently, there exists $q\in T\setminus R$ such that $T=R[q]$ and $q^2\in M$.

In each of the above cases, $M=\mathcal{C}(R,T)$.
\end{theorem} 

\section{Recall on  closures involved in ring extensions}

\subsection{Definitions}

The following definitions are needed for our study. 
 We first consider  more or less  classical concepts.
  
\begin{definition}\label{1.3}  An integral extension  $R\subseteq S$  is called  {\it infra-integral} \cite{Pic 2} (resp.; {\it subintegral} \cite{S}) if all its residual extensions are isomorphisms (resp$.$; and is an {\it i-extension}).    
\end{definition} 
  
We next consider some other concepts that will allow us to give an elementwise interpretation of the above definitions. Two of them are classical and the third was developed by the first author and does not seem well known, but is essential in the sequel.
  
  \begin{definition}\label{1.33}
 According to \cite[Definitions 1.1, 1.2 and 1.5, Proposition 1.6]{Pic 0}, \cite{Pic 2} and \cite{S}, and setting $p_r(X):=X^2-rX\in R[X]$, where $R\subseteq S$ is a ring extension, $R\subseteq S$ is called:

\begin{itemize} 
  
\item {\it s-elementary} (resp.; {\it t-elementary, u-elementary }) if $S=R[b]$, where $p_0(b),bp_0(b)\in R$ (resp.; $p_r(b),bp_r(b)\in R$ for some $r\in R,\ p_1(b),bp_1(b)\in R$). 
  \end{itemize} 
 In the following, the letter x denotes s, t or u. 
 \begin{itemize}
\item {\it cx-elementary} if $R\subseteq S$ is a tower of finitely many x-elementary extensions. 
 
\item{\it x-integral} if there exists a directed set $\{S_i\}_{i\in I}\subseteq[R,S]$ such that $R\subseteq S_i$ is cx-elementary and $S=\cup_{i\in I}S_i$.
 \end{itemize} 
 
In particular, an s-elementary (resp.; u-elementary) extension is a t-elementary extension. In the same way, a cs-elementary (resp.; cu-elementary) extension is a ct-elementary extension and an s-integral  (resp.; u-integral) extension is a t-integral extension.
  
\begin{proposition} \label{1.36}
An s-integral (resp.; t-integral) extension is nothing but a subintegral \cite{S} (resp.; infra-integral \cite{Pic 2}) extension.
  \end{proposition}
  
  \begin{definition} \label{closure}
An extension $R\subseteq S$ is called {\it s-closed} (or {\it seminormal}) (resp.; {\it t-closed}, {\it u-closed} (or {\it anodal})) if an element $b\in S$ is in $R$ whenever $p_0(b),bp_0(b)\in R$ (resp.; $p_r(b),bp_r(b)\in R$ for some $r\in R,\ p_1(b),bp_1(b)\in R$).
\end{definition}
 
Let x$\in\{$s,t,u$\}$. The {\it x-closure} ${}_S^xR$ of $R$ in $S$ is the smallest element $B\in [R,S]$ such that $B\subseteq S$ is x-closed and the greatest element $B'\in [R,S]$ such that $R\subseteq B'$ is x-integral. It follows that ${}_S^uR\subseteq{}_S^tR$.  Note that the $s$-closure is actually the seminormalization ${}_S^+R$ of $R$ in $S$ and is the greatest subintegral extension of $R$ in $S$. Note also that the $t$-closure $ {}_S^tR$ is the greatest infra-integral  extension of $R$ in $S$.
  
   In some next subsections, other closures will be defined.

The {\it canonical decomposition} of an arbitrary ring extension $R\subset S$ is $R\subseteq{}_S^+R\subseteq{}_S^tR\subseteq\overline R\subseteq S$.  One of our aims is to investigate when a dual decomposition exists.
   \end{definition}  
 
  \begin{proposition}\label{1.34}An FCP Pr\"ufer extension is $t$-closed.
 \end{proposition}
 
\begin{proof}Let $R\subset S$ be an FCP Pr\"ufer extension. Since an FCP Pr\"ufer extension is a tower of finitely many Pr\"ufer minimal extensions, let $R_0:=R\subset\ldots\subset R_i\subset\ldots\subset R_n:=S$ be a maximal chain; so that, each $R_i\subset R_{i+1}$ is minimal Pr\"ufer for $i\in\{0,\ldots,n-1\}$. Let $b\in S$ such that $b^2-rb, b^3-rb^2\in R\subseteq R_{n-1}$ for some $r\in R\subseteq R_{n-1}$. Then, either $b$ or $b-r$ is in $R_{n-1}$ by \cite[Proposition 3.1]{FO}; so that, $b\in R_{n-1}$. An easy induction shows that $b\in R$ and $R\subset S$ is t-closed.
 \end{proof}
 
   \subsection{ Technical results needed in the sequel}

   The next proposition describes the link between the elements of the canonical decomposition and minimal extensions.

\begin{proposition}\label{1.31} \cite[Proposition 4.5]{Pic 6} Let there be an integral extension $R\subset S$ and a maximal chain  $R=R_0\subset\cdots\subset R_i\subset\cdots\subset R _n= S$ of subextensions,  each $R_i\subset R_{i+1}$ being minimal.  The following statements hold: 

\begin{enumerate}
\item $R\subset S$ is subintegral if and only if each $R_i\subset R_{i+1}$ is  ramified. 

\item $R\subset S$ is seminormal and infra-integral if and only if each  $R_i\subset R_{i+1}$ is decomposed. 

\item  $R\subset S$ is  infra-integral if and only if each  $R_i\subset R_{i+1}$ is either decomposed or ramified. 

\item $R \subset S$ is t-closed if and only if each $R_i\subset R_{i+1}$ is inert. 

\item  $R\subset S$ is  seminormal if and only if each  $R_i\subset R_{i+1}$ is either decomposed or inert.
\end{enumerate}
If either (1) or (4) holds, then $\mathrm {Spec}(S)\to\mathrm{Spec}(R)$ is bijective.
\end{proposition} 

\begin{lemma}\label{1.9} \cite[Corollary 3.2]{DPP2} If there exists a maximal chain $R=R_0 \subset\cdots\subset  R_i \subset\cdots \subset R_n=S$ of extensions, such that each $R_i\subset R_{i+1}$ is minimal, then  $\mathrm{Supp}(S/R)=\{ \mathcal C (R_i, R_{i+1})\cap R\mid i=0,\ldots,n-1\}$.
\end{lemma}

 We emphasize on the following useful result.

\begin{proposition}\label{1.91} \cite[Proposition 2.4]{Pic 15} Let $R\subset S$ be an integral FCP extension. Then, $R\subset S$ is seminormal if and only if $(R:S)$ is an intersection of finitely many maximal ideals of $S$.  

If this condition holds, the intersection giving $(R:S)$ is irredundant in $R$ and $S$.
\end{proposition}

\begin{proposition}\label{SUP} Let $R\subseteq S$ be an FCP integral extension with conductor $I$. Then $\mathrm V_R(I)$ has finitely many elements and is contained in $\mathrm{Max}(R)$; so that, $\mathrm V_R(I)=\mathrm{Supp}(S/R)=\mathrm{MSupp}(S/R)$. Moreover, for each $Q\in\mathrm{Spec}(S)$ such that $R\cap Q\in\mathrm D_R(I)$, the residual extension $\kappa(Q\cap R)\to\kappa(Q)$ is an isomorphism.
\end{proposition}

\begin{proof} Since $R\subset S$ is an integral FCP extension, $R/I$ is an Artinian ring \cite[Theorem 4.2]{DPP2} and the extension is finite; so that,  $\mathrm{V}_R(I)=\mathrm{Supp}(S/R)=\mathrm{MSupp}(S/R)$, where the last equation is valid because $R/I$ is  Artinian. 

Let $Q\in\mathrm{Spec}(S)$ such that $P:=R\cap Q\in\mathrm D(I)$. Since $P\not\in\mathrm{Supp}(S/R)$, we have that $R_P=S_P=S_Q$, whence $\kappa(Q\cap R)\cong\kappa(Q)$.
\end{proof}

 \begin{lemma}\label{SUP1} 
Let $R\subseteq S$ be an FCP integral t-closed extension. Then, for each $M\in\mathrm{MSupp}(S/R)$, there exists a unique $N\in\mathrm{Max}(S)$ lying above $M$. Moreover, $(R_M:S_M)=MS_N=NS_N$ and $\kappa(N)=S_N/(R_M:S_M)$. We denote by ($\dag$) this last statement for a  later use.
\end{lemma}

\begin{proof} 
Let $M\in\mathrm{MSupp}(S/R)$. Since $R\subset S$ is an FCP t-closed extension, so is $R_M\subset S_M$ and $(R_M:S_M)=MR_M\in\mathrm{Max}(S_M)$ by \cite[Proposition 4.10]{DPP2}; so that, $MR_M=MS_M$. Moreover, there exists a unique $N\in\mathrm{Max}(S)$ lying above $M$ by Proposition \ref{1.31} and $S_M=S_N$ according to \cite[Proposition 2, page 40]{Bki A1}. These facts combine to yield that $I:=(R_M:S_M)=MS_N=NS_N$. In particular, $\kappa(M)=R_M/I$ and $\kappa(N)=S_N/NS_N=S_N/I$.
\end{proof}

 \section{Co-integral closures} 

Let $R\subset S$ be an almost-Pr\"ufer FCP extension. Since $\widetilde R$ is the least $T\in[R,S]$ such that $T\subseteq S$ is integral (\cite[Proposition 4.17]{Pic 5}), we can ask if, for any FCP extension $R\subset S$, there exists a least $T\in[R,S]$ such that $T\subseteq S$ is integral. The answer is yes, as we are going to show. But this statement  may not hold for any arbitrary ring extension. To see this we will use for instance \cite[Example 2.1]{DPPS} where the following situation holds. Let $p$ be a prime number, $k$ a field with $\mathrm{c}(k)=p$, and $n$ a prime number such that $n>p$. Let $X$ be an indeterminate over $k$, and set $K:=k(X),\ F:=k(X^n)$ and $L:=(X^n+X^p)$. Then, we get that $k=F\cap L\subset K=FL$ is not algebraic, and then not integral, while $F\subset K$ and $L\subset K$ are minimal field extensions, and then integral. Actually intersecting two minimal integral extension whose targets are $S$ is a first test for the validity of the result we have in view. 

\begin{proposition}\label{0.15} Let $R\subset S$ be an FCP extension and $T,U\in[R,S]$ be such that $T\subset S$ and $U\subset S$ are minimal integral extensions with $T\neq U$. Then, $T\cap U\subset S$ is integral with $T\cap U\neq T,U$. 
\end{proposition}

\begin{proof} Set $V:=T\cap U,\ M:=(T:S)$ and $ N:=(U:S)$.  

Assume first that $M\neq N$. Then, $V\subset S$ is integral according to \cite[Proposition 6.6]{DPPS}, since so are $V\subset T$ and $V\subset U$.

Assume now that $M= N$. It follows by \cite[Proposition 5.7]{DPPS} that $M\in\mathrm{Max}(V)$. But we also have $M\in\mathrm{Max}(T)$ and $M\in\mathrm{Max}(U)$ with $M=(V:S)$. 
It follows  that $V/M\subset T/M$ and $V/M\subset U/M$ are FCP field extensions, and then integral by \cite[Theorem 4.2]{DPP2}, since any minimal subextension of $V/M\subset T/M$ and $V/M\subset U/M$ is inert by Theorem \ref{minimal}. Hence, $V/M\subset S/M$ is an integral extension, and so is $V\subset S$. Obviously, $T\cap U\neq T,U$ since $T\neq U$.
\end{proof}

To get the general result, the following lemmata are needed.

\begin{lemma}\label{0.16} Let $R\subset S$ be an FCP extension and $T,U\in[R,S]$, such that $T\subset S$ is integral and $U\subset S$ is  minimal integral with $U\not\in[T,S]$. Then, there exists $W\in[T\cap U,T]$ such that $W\subset T$ is minimal integral. 
\end{lemma}

\begin{proof} Let $\{T_i\}_{i=0}^n\subseteq [R,S]$ be a maximal chain such that $T_0:=S$ and $T_n:=T$; so that, each $T_{i+1}\subset T_i$ is minimal integral. For each $i\in\mathbb N_n$, set $V_i:=T_i\cap U$, giving $V:=V_n=T\cap U$. Obviously, we have $V_{i+1}=T_{i+1}\cap V_i$. We intend to build by induction on $i\in\mathbb N_n$ a set $\{W_i\}_{i=1}^n\subseteq[V,S]$ such that $W_i\in[V_i,T_i]$ with $W_i\subset T_i$ minimal integral for each $i\in\mathbb N_n$. Our induction hypothesis is the following: for $i\in\mathbb N_n$, there exists $W_i\in[V_i,T_i]$ with $W_i\subset T_i$ minimal integral and either $W_i\neq T_{i+1}$ or $W_i=T_{i+1}\neq V_i$ when $i<n$.

For $i=1$, since $T_1\subset S$ and $U\subset S$ are minimal integral, we know by Proposition \ref{0.15} that $V_1\subset T_1$ is an integral FCP extension. Of course, $V_1\neq T_1$ because $T_1\neq U$. In particular, there exists $W_1\in[V_1,T_1]$ with $W_1\subset T_1$ minimal integral. Assume that $W_1=T_2$ leading to $V_2=T_2\cap U=W_1\cap U\supseteq V_1\cap U=V_1$; so that, $V_2=V_1$. If $T_2=V_1$, then $T\subseteq T_2\subseteq U$ leads to a contradiction. It follows  that $T_2\neq V_1$.

Assume that the induction hypothesis holds for some $i\in\mathbb N_{n-1}$; so that, $W_i\subset T_i$ is minimal integral for some $W_i\in[V_i,T_i]$  such that either $W_i\neq T_{i+1}$ or $W_i=T_{i+1}\neq V_i$.  

First, assume that $W_i\neq T_{i+1}$ and set $W'_{i+1}:=T_{i+1}\cap W_i$. We claim that $W'_{i+1}\neq T_{i+1}$. Otherwise, $T_{i+1}\subseteq W _i$, which leads to $T_{i+1}= W_i$, a contradiction. Then, $W'_{i+1}\subset T_{i+1}$ is integral by Proposition \ref{0.15} and there exists $W_{i+1}\in[W'_{i+1},T_{i+1}]$ such that $W_{i+1}\subset T_{i+1}$ is minimal integral. If $i=n-1$, we are done. If $i<n-1$, assume that $W_{i+1}=T_{i+2}$. Then, $V_{i+2}=T_{i+2}\cap U=W_{i+1}\cap U\supseteq V_{i+1}\cap U=V_{i+1}$; so that, $V_{i+2}=V_{i+1}$. If $T_{i+2}=V_{i+1}$, then $T\subseteq T_{i+2}\subseteq U$ leads to a contradiction. It follows  that $T_{i+2}\neq V_{i+1}$. Then, the induction hypothesis holds. We therefore  have the following commutative diagram:
\centerline{$\begin{matrix}
     V_i     & \to &     W_i    &  {} &     \to      & {}   &     T_i       \\
\uparrow &  {} & \uparrow &  {}  &     {}       & {}  &  \uparrow  \\
 V_{i+1}  & \to & W'_{i+1} & \to & W_{i+1} & \to &   T_{i+1}       
\end{matrix}$}
Now, assume that $W_i=T_{i+1}\supset V_i$. It follows that $V_i\subseteq W_i\cap U=T_{i+1}\cap U=V_{i+1}\subseteq V_i$, which gives $V_{i+1}=V_i$ and $V_{i+1}\subset T_{i+1}$. Then, there exists $W_{i+1}\in[V_{i+1},T_{i+1}]$ such that $W_{i+1}\subset T_{i+1}$ is minimal integral. If $i= n-1$, we are done. If $i<n-2$, assume that $W_{i+1}=T_{i+2}=V_{i+1}$, then $T\subseteq T_{i+2}\subseteq U$ leads to a contradiction. It follows that $T_{i+2}\neq V_{i+1}$. Then, the induction hypothesis holds for $i+1$ and then for any $i\in\mathbb N_n$. In particular, there exists $W_n\in[V_n,T_n]\setminus\{T_i\}_{i=0}^n$ such that $W_n\subset T_n$ is minimal integral, that is there exists $W\in[T\cap U,T]$ such that $W\subset T$ is minimal integral.
\end{proof}

\begin{lemma}\label{0.17} Let $R\subset S$ be an FCP extension and $T,U\in[R,S]$, such that $T\subset S$ is integral and $U\subset S$ is  minimal integral, with $U\not\in[T,S]$. Then, $T\cap U\subset S$ is an integral extension. 
\end{lemma}

\begin{proof} Set $V:=T\cap U$. We are going to show by induction that there exists a maximal finite chain $\{W_j\}_{j=0}^m$ such that $W_0:=T$ and $W_m:=V$, where $W_{j+1}\subset W_j$ is minimal integral. Our induction hypothesis is the following: for $k\in\{0,\ldots,m-1\}$, where $m$ is some positive integer such that $m\leq\ell[T\cap U,T]$, there exists a maximal finite chain $\{W_j\}_{j=0}^k\subseteq[V,T]$ such that $W_0:=T$ and $W_{j+1}\subset W_j$ is minimal integral for any $j\in\{0,\ldots,k\}$. The induction hypothesis holds for $k=0$ by Lemma \ref{0.16}. Assume that the induction hypothesis holds for some $k\in\{0,\ldots,m-1\}$. Set $T':=W_k\in[V,T]$; so that, $T'\subset S$ is an integral extension. Moreover, $V\subseteq T'\cap U\subseteq T\cap U=V$; so that, $T'\cap U=V$. 
 If $U\in[T',S]$, then $T'\cap U=T'=V$, and $V\subset S$ is integral. Assume now that $U\not\in[T',S]$.
We apply Lemma \ref{0.16} to the extensions $T'\subset S$ which is integral and $U\subset S$ which is minimal integral. 
Then, there exists $W_{k+1}\in[V,T']\subseteq[V,T]$ such that $W_{k+1}\subset T'=W_k$ is minimal integral. It follows that  the induction hypothesis holds for $k+1$, and we have a decreasing maximal chain in $[V,T]$ which stops at some $W_m=V$ since $R\subset S$ has FCP. In particular, $V\subset T$ is an integral extension, and so is $V\subset S$. 
\end{proof}

\begin{proposition}\label{0.18} Let $R\subset S$ be an FCP extension and $T,U\in[R,S]$ such that $T\subset S$ and $U\subset S$ are integral. Then, $T\cap U\subset S$ is  integral. 
\end{proposition}

\begin{proof} Set $V:=T\cap U$. If $V\in\{T,U\}$, we have the result. Now, assume that $V\neq T,U$. Let $\{U_i\}_{i=0}^n$ be a maximal chain such that $U_0:=S$ and $U_n:=U$. Then, $U_{i+1}\subset U_i$ is minimal integral for each $i\in\{0,\ldots,n\}$. Set $V_i:=T\cap U_i$ for each $i\in\{0,\ldots,n\}$. We are going to prove by induction on $i$, that $V_i\subseteq S$ is integral for any $i\in\{0,\ldots,n\}$. The induction hypothesis holds obviously for $i=0$ and is Lemma \ref{0.17} for $i=1$ 
 if $U_1\not\in[T,S]$. If $U_1\in[T,S]$, then $V_1=T$; so that, $V_1\subset S$ is  integral.
Assume that the induction hypothesis holds for some $i\in\{0,\ldots,n-1\}$, that is $V_i\subseteq S$ is integral. We have $V_{i+1}=T\cap U_{i+1}=T\cap U_i\cap U_{i+1}=V_i\cap U_{i+1}$. Since $V_i\subseteq S$ is integral, so is $V_i\subseteq U_i$. Moreover, $U_{i+1}\subset U_i$ is minimal integral. Using again Lemma \ref{0.17} for the extensions $V_i\subseteq U_i$ and $U_{i+1}\subset U_i$  if $U_{i+1}\not\in[V_i,U_i]$,  
we get that $V_{i+1}\subset V_i$ is integral, and so is $V_{i+1}\subset S$. 
  If $U_{i+1}\in[V_i,U_i]$, then $V_{i+1}=U_{i+1}\cap V_i=V_i$, and then the result also holds.
 The induction hypothesis holds  for $i+1$, and then for any $i\in\{0,\ldots,n\}$, and in particular for $i=n$; so that, $V\subset S$ is integral.
\end{proof}

The previous results combine to yield that, for an FCP extension $R\subset S$, we can introduce a new closure, namely  the {\it co-integral closure} as it is shown in the following Theorem.
 
\begin{theorem}\label{0.19} Let $R\subset S$ be an FCP extension. There exists a least $T\in[R,S]$ such that $T\subseteq S$ is integral. We call such $T$ the co-integral closure of $R\subset S$ and denote it by $\underline R$. Moreover, $\underline R=\cap[U\in [R,S]\mid U\subseteq S$ integral$\}$  and $\widetilde{R}\subseteq \underline R$. 
\end{theorem}

\begin{proof} Set $\mathcal E:=\{T\in[R,S]\mid T\subseteq S$ integral$\}$. Then $S\in\mathcal E$, and, since $[R,S]$ is an Artinian lattice, $\mathcal E$ admits a minimal element. We claim that this minimal element is unique. Indeed, assume that there exist $T,U$ two distinct minimal elements of $\mathcal E$; so that, $T\neq U$ where $T\subset S$ and $U\subset S$ are integral extensions. Then, $T\cap U\subset S$ is integral according to Proposition \ref{0.18}, a contradiction with the minimality of $T$. It follows that the minimal element of $\mathcal E$, that we denote $\underline R$, is the least $T\in[R,S]$ such that $T\subseteq S$ is integral. In particular, $\underline R=\cap[U,\in[R,S]\mid U\subseteq S$ integral$\}$  and any $T\in[\underline R,S]$ is such that $T\subseteq S$ is integral.

Now, we show that $\widetilde{R}\subseteq\underline R$. Let $V$ be the integral closure of $R$ in $\underline R$ and set $W:=\widetilde{R}V$. We have the following commutative diagram:
\centerline{$\begin{matrix}
\widetilde{R} & \to & W=\widetilde{R}V & \to &        S          \\
   \uparrow    &  {} &        \uparrow        &  {} &   \uparrow    \\
         R         & \to &              V             & \to & \underline R       
\end{matrix}$}
Since $R\subseteq V$ is integral, so is $\widetilde{R}\subseteq W$ and $R\subseteq W$ is almost-Pr\"ufer. According to \cite[Proposition 4.16]{Pic 5}, $R_P\subseteq W_P$ is either integral $(*)$ or Pr\"ufer $(**)$ for each $P\in\mathrm{Spec}(R)$. In case $(*)$, we get $R_P=(\widetilde{R})_P$; so that, $V_P=W_P$. In case $(**)$, we get $R_P=V_P$; so that, $(\widetilde{R})_P=W_P$, leading to $V_P\subseteq W_P$ is Pr\"ufer. Then, $V\subseteq W$ is Pr\"ufer, as is $V\subseteq\underline R$, and so is  $V\subseteq W\underline R$ by \cite[Corollary 5.11, page 53]{KZ}. Since $V\subseteq\underline R\subseteq W\underline R$ holds, we deduce that $\underline R\subseteq W\underline R$ is Pr\"ufer. But, $W\underline R\in[\underline R,S]$ implies that $\underline R\subseteq W\underline R$ is integral, giving $\underline R=W\underline R$. To end, $\widetilde{R} \subseteq W\subseteq W\underline R= \underline R$.
\end{proof}

\begin{corollary}\label{0.20} Let $R\subset S$ be an FCP extension. Let $T$ be  the co-integral closure of $R\subset S$ and $U$ be the integral closure of $R$ in $T$. Then, $U\subseteq S$ is almost-Pr\"ufer, and $T$ is the Pr\"ufer hull of $U\subseteq S$.
\end{corollary}

\begin{proof} Since $U\subseteq T$ is Pr\"ufer and $T\subseteq S$ is integral, we get that $U\subseteq S$ is almost-Pr\"ufer, and $T$ is the Pr\"ufer hull of $U\subseteq S$.
\end{proof}

\begin{theorem}\label{0.21} Let $R\subset S$ be an FCP extension. Then, $R\subseteq S$ is almost-Pr\"ufer if and only if $\widetilde{R}=\underline R$.
\end{theorem}

\begin{proof} Assume that $\widetilde{R}=\underline R$. Then, $R\subseteq\widetilde{R}$ is Pr\"ufer and $\widetilde{R}=\underline R\subseteq S$ is integral; so that, $R\subseteq S$ is almost-Pr\"ufer.

Conversely, assume that $R\subseteq S$ is almost-Pr\"ufer. According to \cite[Proposition 4.17]{Pic 5}, $\widetilde{R}$ is the least $T\in[R,S]$ such that $T\subseteq S$ is integral, which implies that $\widetilde{R}= \underline R $.
\end{proof}

The previous Theorem shows that for an FCP extension $R\subseteq S$ which is not almost-Pr\"ufer, then $\widetilde{R}\subset\underline R $. In particular, $R\subset\underline R $ is not Pr\"ufer, and then not integrally closed, because an integrally closed FCP extension is Pr\"ufer. This situation holds in Example \ref{0.23}.

If $R\subseteq S$ is an almost-Pr\"ufer extension, the closures $\overline R$ and $\underline R$ are linked as we can see in the following Proposition which improves \cite[Theorem 6.13]{Pic 15},  since $\widetilde{R}= \underline R $. 

\begin{proposition}\label{0.22} If $R\subset S$ is an almost-Pr\"ufer FCP extension, there are order-isomorphisms: 
\begin{enumerate}
\item $\varphi:[R,\overline R]\to[\underline R,S]$ defined by $\varphi(T):=T\underline R$ for any $T\in[R,\overline R]$, and whose inverse $\psi:[\underline R,S]\to[R,\overline R]$ is defined by $\psi(U):=U\cap\overline R$ for any $U\in[\underline R,S]$.

\item $\theta:[R,\widetilde R]\to[\overline R,S]$ defined by $\theta(T):=T\overline R$ for any $T\in[R,\widetilde R]$, and whose inverse $\rho:[\overline R,S]\to[R,\widetilde R]$ is defined by $\rho(U):=U\cap\widetilde R$ for any $U\in[\overline R,S]$.

\item $\gamma:[R,S]\to[R,\overline R]\times[R,\widetilde R]$ defined by $\gamma(W):=(W\cap\overline R,W\cap\widetilde R)$ for any $W\in[R,S]$, and whose inverse $\delta:[R,\overline R]\times[R,\widetilde R]\to[R,S]$ is defined by $\delta((W,W')):=WW'$ for any $(W,W')\in[R,\overline R]\times[R,\widetilde R]$.
\end{enumerate}
\end{proposition}

\begin{proof} Since $R\subset S$ is an almost-Pr\"ufer FCP extension, \cite[Proposition 4.3]{Pic 15} says that $R\subset S$ splits at $\widetilde R$ and $\overline R$ with $\widetilde R=(\overline R)^o$ 
 and Theorem \ref{0.21} asserts that $\underline R=\widetilde R$.

(1) According to \cite[Theorem 6.13 (4)]{Pic 15}, the map $\varphi:[R,\overline R]\to[\underline R,S]$ is an order-isomorphism. Moreover, $\psi:[\underline R,S]\to[R,\overline R]$ 
defined by $\psi(U):=U\cap \overline R$
 clearly exists and preserves order. Let $T\in[R,\overline R]$ and $U\in[\underline R,S]$. Then, 
 $T\cap\widetilde{R}\subseteq\overline R\cap\widetilde{R}=R$ and $S=\widetilde R\overline R\subseteq U\overline R$. It follows that $T=T\underline R\cap\overline R=\psi[\varphi(T)]$ and $U=(U\cap\overline R)\underline R=\varphi[\psi(U)]$, a consequence of \cite[Proposition 4.21]{Pic 5}. Then $\psi$ is an order-isomorphism such that $\psi={\varphi}^{-1}$.

(2) Similarly $\theta:[R,\widetilde R]\to[\overline R,S]$ is an order-isomorphism by \cite[Theorem 6.13 (3) ]{Pic 15} and $\rho:[\overline R,S]\to[R,\widetilde R]$ defined by $\rho(U):=U\cap \widetilde R$ exists and preserves order. Let $T\in[R,\widetilde R]$ and $U\in[\overline R,S]$. Then, $T\cap\overline{R}\subseteq\overline R\cap\widetilde{R}=R$ and $S=\widetilde R\overline R\subseteq U\widetilde R$. It follows that $U=(U\cap\widetilde R)\overline R=\theta[\rho(U)]$, a consequence of \cite[Proposition 4.21]{Pic 5}. Now, \cite[Corollary 4.3]{Pic 5} implies that $T=T\overline R\cap\widetilde R=\rho[\theta(T)]$.
 Then  $\rho$ is an  order-isomorphism such that $\rho={\theta}^{-1}$.

(3) According to \cite[Theorem 6.13 (5)]{Pic 15}, $\delta':[R,\widetilde R]\times[R,\overline R]\to[R,S]$ defined by $\delta'((W,W')):=WW'$ for any $(W,W')\in[R,\widetilde R]\times[R,\overline R]$ is an order-isomorphism, and so is obviously $\delta:[R,\overline R]\times[R,\widetilde R]\to[R,S]$ defined by $\delta((W,W')):=WW'$ for any $(W,W')\in[R,\overline R]\times[R,\widetilde R]$. Moreover, the map $\gamma:[R,S]\to[R,\overline R]\times[R,\widetilde R]$ defined by $\gamma(W):=(W\cap\overline R,W\cap\widetilde R)$ for any $W\in[R,S]$ exists and preserves order. 

Let $W\in [R,S]$. Then, $U:=\delta[\gamma(W)]=(W\cap\overline R)(W\cap\widetilde R)\ (*)$. Taking into account \cite[Proposition 4.18]{Pic 5} and localizing $(*)$ at any $M\in\mathrm{MSupp}(S/R)=\mathrm{MSupp}(\overline R/R)\cup\mathrm{MSupp}(\widetilde R/R)$, with $\mathrm{MSupp}(\overline R/R)\cap\mathrm{MSupp}(\widetilde R/R)=\emptyset$, we get the following:

If $M\in\mathrm{MSupp}(\overline R/R)$, then $R_M=(\widetilde R)_M$ and $S_M=(\overline R)_M$; so that, $U_M=W_M$. If $M\in\mathrm{MSupp}(\widetilde R/R)$, then $R_M=(\overline R)_M$ and $S_M=(\widetilde R)_M$; so that, $U_M=W_M$. It follows that $U=W$ and $\delta[\gamma(W)]=W$ for any $W\in [R,S]$. Then, $\delta\circ\gamma$ is the identity on $[R,S]$. Since $\delta$ is an order-isomorphism, so is $\gamma=\delta^{-1}$. 
 \end{proof}

In  the context of algebraic orders, we built an  FIP extension $R\subset K$ which is not almost-Pr\"ufer \cite[Example 6.15]{Pic 15}, but where we exhibit  its almost-Pr\"ufer closure. As the co-integral closure was not still defined, we recall this example here and  show that  $K$ is the co-integral closure of the extension.

\begin{example}\label{0.23} \cite[Example 6.15]{Pic 15} Let $K$ be an algebraic number field, $T'$ be its ring of algebraic integers, and $R'$ be an algebraic order of $K$ such that $R'\subset T'$ is a minimal inert extension. Set $M':=(R':T')\in\mathrm{Max}(R')$; so that, $M'\in\mathrm{Max}(T')$. Let $N'\in\mathrm{Max}(R'),\ M'\neq N'$. Set $\Sigma:=R'\setminus(M'\cup N')$, which is a multiplicatively closed subset. Then, $R:=R'_{\Sigma}$ is a semilocal ring with two maximal ideals $M:=M'_{\Sigma}$ and $N:=N'_{\Sigma}$. Moreover $T:=T'_{\Sigma}$ is also a semilocal ring with two maximal ideals, and such that $R\subset T$ is a minimal inert extension with $M:=M'_{\Sigma}=(R:T)\in\mathrm{Max}(T)$. 
The other maximal ideal of $T$ is $P:=NT$. 
 In particular, $T$ is a Pr\"ufer domain with quotient field $K$,  
  $T\subset T_M$ is minimal Pr\"ufer with $P=\mathcal{C}(T,T_M)$ and 
 $\{N\}=\mathrm{MSupp}_R(T_M/T)$.  
 Moreover, $T$ is the integral closure of $R\subset K$, and $R\subset K$ is an FCP extension.
There is a commutative diagram with 
\noindent $[R,K]=\{R,T,R_M,T_M,T_P,K\}$:
\centerline{$\begin{matrix}
 {} &       {}       & R_M &      {}       &   {}   &      {}       & {} \\
 {} & \nearrow &    {}   & \searrow &   {}   &      {}       & {} \\
 R &      {}       &    {}   &       {}      & T_M &      \to      & K \\
 {} & \searrow &   {}   & \nearrow &   {}    & \nearrow & {} \\
 {} &      {}       &   T   &       \to      & T_P  &      {}       & {}
\end{matrix}$}
where $R\subset T_M$ is almost-Pr\"ufer with $R_M\subset T_M$  minimal inert as $R\subset T$, and $R\subset R_M$ is Pr\"ufer minimal as $T\subset T_M,\ T_M\subset K,\ T\subset T_P$ and $T_P\subset K$. 

When considering the extension $R\subset K$, obviously, there does not exist some $V\in[R,K[$ such that $V\subset K$ is integral, which implies that $K$ is the co-integral closure of $R\subset K$. And $R_M$ is the Pr\"ufer hull of $R\subset K$, and also of $R\subset T_M$.
 Then, as it is said after Theorem \ref{0.21}, the co-integral closure and  the Pr\"ufer hull do not always coincide. But, in the present example, as $R\subset T_M$ is almost-Pr\"ufer, $R_M$ is also the co-integral closure of $R\subset T_M$.
\end{example} 

\begin{example}\label{0.24}
Let $R\subseteq S$ be an extension. Each $T\in[R,S]$ such that $T\subseteq S$ is integrally closed contains $\overline R$. Dually, things are not so simple. For an extension $R\subseteq S$, we set $[R,S]_{IC}:=\{T \in[R,S]\mid T\cap\overline R=R\}$, {\it i.e.} $T\in [R,S]_{IC}$ if and only if $R\subseteq T$ is integrally closed. In particular, $\widetilde R\in [R,S]_{IC}$. Clearly $R$ belongs to $[R,S]_{IC}$ and this set is $\cup$-inductive in the poset $([R,S];\subseteq )$. By Zorn Lemma, it contains maximal elements. It also contains $T \in[R,S]$ such that  $T\setminus R$ is a mcs by \cite[Th\'eor\`eme 1(d)]{SAM}. 
 
Let $\mathcal M$ be a maximal element of $[R,S]_{IC}$ and $s\in S\setminus\mathcal M$. Then $\mathcal M\subset\mathcal M[s]$ cannot be integrally closed because otherwise $R\subseteq\mathcal M[s]$ would be integrally closed. It follows that there is some $p(s)\in\mathcal M[s]\setminus\mathcal M$ which is integral over $\mathcal M$, and then $s$ is algebraic over $\mathcal M$. Therefore, $S$ is algebraic over $\mathcal M$. If we choose some maximal $\mathcal M$, containing $\widetilde R$, we get a tower $R\subseteq\widetilde R\subseteq\mathcal M\subseteq S$, where the last extension is algebraic. In particular, for each $s\in S$, there is some $m\in\mathcal M$ such that $ms$ is integral over $\mathcal M$.
 
There exist some extensions $R\subseteq S$ such that $[R,S]_{IC}$ has a greatest element, for example if the extension is chained. Other examples are quasi-Pr\"ufer extensions like FCP extensions, because such extensions are integrally closed if and only if they are Pr\"ufer and subextensions of quasi-Pr\"ufer extensions are quasi-Pr\"ufer \cite[Corollary 3.3]{Pic 5}. In this case the greatest element of $[R,S]_{IC}$ is $\widetilde R$. Actually, $\widetilde R\subseteq S$, is quasi-Pr\"ufer, whence a $\mathcal P$-extension; that is, each element $s\in S$ is a zero of a polynomial $p(X)$ over $\widetilde R$, whose content is $\widetilde R$. In particular $S$ is algebraic over  $\widetilde R$.  
 
Another example is given by ``completely" distributive extensions $R\subseteq S$ (extensions such that for each $V\in[R,S]$ and for any family $\{W_i\}_{i\in I}$ of elements of $[R,S]$, we have $(\prod_{i\in I}W_i)\cap V=\prod_{i\in I}(W_i\cap V)$). 
 The greatest element of $[R,S]_{IC}$ is the product of all elements of $[R,S]_{IC}$. Denote by $\{T_i | i\in I\}$ the elements of $[R,S]_{IC}$, by $U$ their product and let $T$ be one of them. Then it is enough to show that $U$ is integrally closed over $R$, because $T\subseteq U$. But we have $U\cap\overline R=\prod_{i\in I}(T_i\cap\overline R)=R$, which proves that $U$ is integrally closed over $R$.
  
Note that the same proof holds if $R\subseteq S$ is distributive and  has FCP (whence has FIP by \cite[Lemma 2.10]{Pic 10}).
 
If $[R,S]_{IC}=\{T_1,\ldots,T_n\}$, we thus get that $\widetilde R=T_1\cdots T_n$.
\end{example}
 
We can also consider for an extension $R\subseteq S$, the set of all $T\in [R,S]$ such that $R\subseteq T$ is seminormal (resp.; t-closed, u-closed). 
An element $T\in [R,S]$ is such that $R\subseteq T$ is seminormal (resp.; t-closed, u-closed) if and only if ${}^+_SR\cap T=R$ (resp.; ${}^t_SR\cap T;{}^u_SR\cap T=R$). Then most of the preceding considerations hold. 

\section{Radicial  and separable  ring extensions}

We first introduce some notation. If $\Pi$ is a property of field extensions, we say that a ring extension is a $\kappa$-$\Pi$-extension if all its residual field extensions verify $\Pi$. In the sequel of the section, we will consider the separable and radicial properties, where in this paper a purely inseparable field extension is called {\it radicial}. 

 \begin{example} \label{CAT} An FCP t-closed extension $R\subset S$ is $\kappa$-catenarian if and only if $R\subset S$ is catenarian. Indeed, according to \cite[Proposition 4.3]{Pic 12}, $R\subset S$ is catenarian if and only if $R_M\subset S_M$ is catenarian for each $M\in\mathrm{MSupp}(S/R)$. Since $R\subset S$ is an FCP t-closed extension, we infer from  property ($\dag$) of Lemma \ref{SUP1} that for each $M\in\mathrm{MSupp}(S/R)$, there exists a unique $N\in\mathrm{Max}(S)$ lying above $M$. Moreover, $(R_M:S_M)=MR_M=NS_N=NS_M$ and $\kappa(N)=S_N/(R_M:S_M)$. Set  $I:=(R_M:S_M)$.  In particular, $\kappa(M)=R_M/I$ and $\kappa(N)=S_N/I$. Then, \cite[Corollary 3.7]{Pic 12} shows that $R_M\subset S_M$ is catenarian  for each $M\in\mathrm{MSupp}(S/R)\Leftrightarrow R_M/I\subset S_M/I$ is catenarian for each $M\in\mathrm{MSupp}(S/R)\Leftrightarrow\kappa(M)\subset\kappa(N)$ is catenarian for each $M\in\mathrm{MSupp}(S/R)\Leftrightarrow\kappa(N\cap R)\subseteq\kappa(N)$ is catenarian for each $N\in\mathrm{Max}(S)$ (because of the bijection $\mathrm{Spec}(S)\to\mathrm{Spec}(R)$) $\Leftrightarrow R\subset S$ is $\kappa$-catenarian.
\end{example}

Kubota introduced the radicial closure of a ring extension, focussing his results on integral ring extensions \cite{KU}. Later on Manaresi  reintroduced this notion and named it weak normalization \cite{MA}. We keep the notation of Manaresi (but not her terminology): for a ring extension $R\subseteq S$, that we will suppose integral for the rest of the Section, we denote by ${}^*_SR$ the radicial closure of $R$ in $S$. We recall from \cite{MA} or \cite{KU}, that ${}^*_SR:=\{x\in S\mid x\otimes 1-1\otimes x\in\mathrm{Nil}(S\otimes_RS)\}$. Moreover, ${}^*_SR$ is the  greatest subextension $R\subseteq T$ of $R\subseteq S$, such that $R\subseteq T$ is a radicial extension; that is, the ring morphism $R\to T$ is universally an $i$-morphism, or equivalently, an $i$-morphism which is $\kappa$-radicial \cite[Definition 3.7.2, p. 248]{EGA} and \cite[Proposition 1]{KU}. 
 For instance, a subintegral extension is radicial. 
 
Yanagihara proved that a ring extension $R\subseteq S$ is radicial if and only if $S$ is the filtered union of all $U\in[R,S]$ such that $R\subseteq U$ is a finite tower of finitely many subextensions of the type $T\subseteq T[x]$, with $x\in S$ and $T\in[R,S]$, such that either $x^p\in T$ and $px\in T$ for some prime integer $p$ or $x^2, x^3 \in T$ \cite[Theorem 2]{YA}. 

 Let $R\subset S$ be an integral ring extension. In \cite{Pic 14}, the second author says that $R\subset S$ is {\it weakly infra-integral} if it is $\kappa$-radicial.  
   She defined the {\it strong t-closure}, that is here termed the $\kappa$-radicial closure, of a ring extension $R\subseteq S$ as the greatest subalgebra $V:={}^\circ_SR\in[R,S]$ such that $R\subseteq V$ is $\kappa$-radicial \cite[Definition 3.1]{Pic 14}. 
   In particular, ${}^*_RS\subseteq {}^\circ_SR$. 

Let $f:R \to S$ be a ring morphism and its co-diagonal ring morphism $\bigtriangledown_f:S\otimes_RS \to S$.
Denote by $I$ the kernel of the surjective ring morphism $\bigtriangledown_f$. If $f$ is of finite type as an algebra, then $I$ is an ideal of finite type and  $I=I^2$ if and only if $I$ is generated by an  idempotent. 
We recall that a ring morphism $f:R\to S$ is called {\it formally unramified} if $\Omega_R(S):=I/I^2$, the $S$-module of differentials of $f$, is zero and {\it unramified} if it is formally unramified and of finite type.  We see that a ring morphism of finite type $f:R\to S$ is unramified  if and only if $I$ is generated by an idempotent.  In this case,   $f$ is called separable by some authors because   $\bigtriangledown_f$ defines $S$ as an  $S\otimes_RS$-projective module, \cite[Proposition 9, p. 32, Chapitre III]{RAY}.  We will not use this terminology.
 But in this paper we will say that an integral extension $R\subset S$ is {\it $\kappa$-separable} if its residual extensions are separable, in which case $R\subseteq U$ and $U\subseteq S$ are $\kappa$-separable for any $U\in [R,S]$. To check that an integral FCP extension $R\subseteq S$ with conductor $I$ is $\kappa$-separable it is enough to verify that the extensions $\kappa(Q\cap R)\to\kappa (Q)$ are separable for each maximal ideal $Q\in\mathrm{V}_S(I)$ by Proposition \ref{SUP}. Note that $R\subseteq U$ is $\kappa$-separable for any $U\in[R,{}^t_SR]$. Note also that a $\kappa$-radicial and $\kappa$-separable extension is infra-integral.
 
The following Proposition belongs to the algebraic geometry folklore, but we can not give any reference. So we give an outline  of its proof.

\begin{proposition}\label{SEP} A finite type extension $R\subseteq S$  is a ring epimorphism if and only if it is radicial and  
unramified. 
Therefore, an integral FCP extension is trivial when it is radicial and unramified.
\end{proposition}
\begin{proof} One implication is clear because an epimorphism is  an $i$-extension and the residual extensions are isomorphisms \cite[Proposition 1.5, page 109]{L}. Now if $R\subseteq S$ is radicial and unramified, then the kernel of $\bigtriangledown:S\otimes_RS\to S$ is generated by an idempotent which is nilpotent. Therefore $\bigtriangledown$ is an isomorphism and the extension is an epimorphism \cite[Proposition 1.5, p.109]{L}.
The last assertion is a consequence of the following facts: each subextension of an integral FCP extension is (module) finite and a module finite ring epimorphism  is surjective \cite[Proposition 1.7, p.111]{L}.
\end{proof} 

We have the following results that may not hold if the extension is not FCP:

\begin{proposition}\label{RAM} Let $R\subseteq S$ be an integral FCP extension. The following statements are equivalent:

\begin{enumerate}

\item The extension is unramified;

\item For all $Q\in\mathrm{Spec}(S)$, the ring morphism $R_{R\cap Q}\to S_Q$ is  formally unramified;

\item $QS_Q =PS_Q$ and $\kappa(P)\to\kappa (Q)$ is a finite separable field extension  for each $Q\in \mathrm{Spec}(S)$ and $P:=Q\cap R$; 

\item The extension is $\kappa$-separable and for each $Q\in\mathrm{Spec}(S)$ and $P:=Q\cap R$, the ideal $PS_P$ of $S_P$ is semiprime;

\item $\kappa(P)\to\kappa(P)\otimes_RS$ is unramified for each $P\in \mathrm{Spec}(R)$.

\end{enumerate} 

It follows that an unramified extension is universally $\kappa$-separable.
\end{proposition}

\begin{proof} (1) $\Leftrightarrow$ (2) is \cite[Exercise 1.14 (vi)]{I}. To prove that (1) $\Leftrightarrow$ (4), we use \cite[Exercise (1), p.38]{RAY} which states that a ring morphism of finite type $R\to S$ is unramified if and only if $QS_Q=PS_Q$ and $\kappa(P)\to\kappa(Q)$ is a finite separable field extension for each $Q\in\mathrm{Spec}(S)$ and $P:=Q\cap R$, which proves (1) $\Leftrightarrow$ (3). We next prove that (3) $\Leftrightarrow$ (4). Suppose that the extension is unramified, we can assume that $R$ is local with maximal ideal $P$. Then $R/P\subseteq S/PS$ is \'etale over the field $R/P$. It follows that $S/PS$ is reduced, and $PS$ is semiprime by \cite[Proposition 1, page 74]{RAY}.
 To prove the converse, if $PS$ is semiprime, we observe that $PS$ is the intersection of the (finitely many) maximal ideals of $S$ lying over $P$, because $R\subseteq S$ is finite. It is now easy to show that $PS_Q= QS_Q$.

Now (1) $\Leftrightarrow$ (5) because of \cite[Property 2.2.vi, page 8]{I} and the fact that the unramified property is stable under any base change. 
\end{proof}

\begin{corollary}\label{RAM1} Let $R\subseteq S$ be an integral FCP t-closed extension. Then, $R\subseteq S$ is unramified if and only if $R\subseteq S$ is $\kappa$-separable.
\end{corollary}

\begin{proof} According to Proposition \ref{RAM}, $R\subseteq S$ is unramified if and only if $QS_Q =PS_Q$ and $\kappa(P)\to\kappa (Q)$ is a finite separable field extension  for each $Q\in \mathrm{Spec}(S)$ and $P:=Q\cap R$.  But the condition ($\dag$) of Lemma \ref{SUP1} shows that $QS_Q=PS_Q$ for each $Q\in\mathrm{Spec}(S)$ and $P:=Q\cap R$. Then, $R\subseteq S$ is unramified $\Leftrightarrow\kappa(P)\to\kappa (Q)$ is a finite separable field extension for each $Q\in\mathrm{Spec}(S)$ and $P:=Q\cap R\Leftrightarrow  R\subseteq S$ is $\kappa$-separable.
\end{proof}

The following facts are known. If $R\subset S\subset T$ is an unramified tower of extension, then $T\subset S$ is unramified and a tower of unramified extensions is unramified. But $R\subset S$ does not need to be unramified, except if $R$ is a field \cite[Lemme 3.4, p.11]{I}. However, we can provide the following result that can be used on a tower of finitely many minimal extensions, the last one being pure. Recall that a ring morphism $f:R\to S$ is called {\it pure} if $R'\to R'\otimes_RS$ is injective for each ring morphism $R\to R'$ (or equivalently, $M\to M\otimes_R S$ is injective for each $R$-module $M$).  
If $f$ is pure, then $f^{-1}(IS)=I$ for each ideal $I$ of $R$; so that, $f$ is injective. Moreover, a pure ring morphism $R\to S$ verifies the condition (O), {\em i.e.} an $R$-module $M$ is zero if $M\otimes_RS = 0$.

\begin{proposition} \label{1.111}Let $R\subset  S \subset T$  be an unramified FCP extension, such that $S\subset T$ is pure. Then $R\subset S$ is unramified.
\end{proposition}  

\begin{proof}Consider a prime ideal $P$ of $R$ and the fiber morphism $\kappa (P) \to \kappa (P)\otimes_RS \to \kappa(P)\otimes_R T$. The last morphism is injective because deduced from the pure morphism $S\subset T$ by the base change $S\to \kappa (P)\otimes_RS$. Now it is enough to use Proposition \ref{RAM}(5), because over a field a subextension of an unramified extension is unramified \cite[Lemma 3.4, p.11]{I}.
\end{proof}

The next result gives a necessary and sufficient condition in order that    the above result holds.

\begin{proposition} \label{1.11} Let $R\subseteq S$ be an integral seminormal FCP extension and $T\in[R,S]$. Then $R\subseteq S$ is unramified if and only if $R\subseteq T$ and $T\subseteq S$ are unramified.
\end{proposition} 
\begin{proof} We need only to prove that if $R\subseteq S$  is unramified then so is $R\subseteq T$. Indeed the unramified property is stable under composition and stable under left division, without any other hypothesis by \cite[Properties 2.2.ii and 2.2.v, p. 8]{I}. Let $Q\in\mathrm{Spec}(T),\ P:=Q\cap R$ and let $M\in\mathrm{Spec}(S)$ be lying over $Q$; so that, $P=R\cap M$. Since $\kappa(P)\to\kappa(M)$ is a separable field extension, so is $\kappa(P)\to\kappa(Q)$. Then, according to Proposition \ref{RAM}, to show that $R\subseteq T$ is unramified, it is enough to show that $PT_P$ is semiprime in $T_P$, that is an intersection of prime ideals of $T_P$. But $R_P\subseteq T_P$ is still a seminormal FCP integral extension with $R_P$ a local ring. It follows that we may assume that $(R,P)$ is a local ring, and we have to show that $PT$ is semiprime in $T$. Set $C:=(R:T)$, which is an intersection of maximal ideals of $T$ and $R$, according to Proposition \ref{1.91}; so that, $C=P$. To conclude, $R\subseteq T$ is unramified.
\end{proof} 

By \cite[Lemma 2.2]{PICET}, if $I$ is the conductor ideal of an extension $R\subseteq S$, then $R\subseteq S$ is unramified if and only if $R/I\subseteq S//I$ is unramified. It follows that an inert minimal extension $R\subset S$ with conductor $M$ is unramified if and only if the field extension $R/M\subset S/M$ is separable, in which case the extension $R\subseteq S$ is called {\it separable inert}. 

\begin{remark} \label{1.12} (1) In the previous Proposition, the condition that $R\subset S$ is seminormal is necessary. Indeed, assume that $R\subset S$ is an integral FCP extension which is not seminormal. Then, $R\neq{}^+_SR$, and there exists $T\in[R,{}^+_SR]$ such that $R\subset T$ is minimal ramified. 
 It follows that $P:=(R:T)\in\mathrm{Max}(R)$. Let $Q\in\mathrm{Max}(T)$ be the (unique because of Theorem \ref{minimal}) maximal ideal of $T$ lying above $P$ and let $M\in\mathrm{Max}(S)$ be lying above $P$, and then above $Q$. Even if $R\subset S$ is unramified; so that, $\kappa(P)\to\kappa (M)$ is   separable by Proposition \ref{RAM}, and then $\kappa(P)\to\kappa(Q)$ is   separable, $R\subset T$ is not unramified. Indeed,  $R_P\subset T_P$ is still minimal ramified with $PR_P=PT_P=PT_Q\subset QT_P=QT_Q$ because $T_P=T_Q$ 
 by \cite[Proposition 2, page 40]{Bki A1}. 
  
(2) The situation of (1) occurs in the following example: Let $(R,M)$ be a SPIR such that $M^2=0$. Set $S:=R^2$ and $T:={}^+_SR=R+(M\times M)$ by \cite[Proposition 2.8]{Pic 9}. Then, $R\subset T$ is minimal ramified and $T\subset S$ is minimal decomposed because $\mathrm{Max}(S)=\{R\times M,M\times R\}$; so that, $R\subset S$ is infra-integral, with residual extensions which are isomorphisms, and then separable. To show that $R\subset S$ is unramified, using Proposition \ref{RAM} (4), it is enough to show that $MS$ is semiprime in $S$, which is obvious since $MS=M\times M=(M\times R)\cap(R\times M)$. Then, $R\subset S$ is unramified although $R\subset T$ is not.

(3) But there exist unramified extensions $R\subset U$, that cannot be seminormal, with $T,S\in]R,U[$ such that $R\subset S,\ S\subset U,\ T\subset U$ are unramified and $R\subset T$ is not unramified.
   
We keep the hypotheses of the preceding example. Let $R\subset V$ be a minimal inert separable extension and set $M:=Rt$. It is always possible to build such an extension. It is enough to choose the SPIR $(R,M)$ such that $k:=R/M$ has a minimal separable field extension $K:=k[x]=k[X]/(f(X))$ and to set $V:=R[X]/(tX,g(X)))$, where $f(X)=\bar g(X)\in k[X]$ is the monic polynomial whose coefficients are the classes of coefficients of $g(X)$. Then, $R\subset V$ is minimal inert separable. Now, set $U:=V\times R$. According to \cite[Proposition III.4]{DMPP}, $S\subset U$ is also minimal inert separable (with $(S:U)=M\times R$, giving the type of the minimal extension); so that, $R\subset U$ is unramified, as the composite of two unramified extensions $R\subset S$ and $S\subset U$. However, we still have $R\subset T$ not unramified and $T\subset U$ unramified, since so are $T\subset S$ and $S\subset U$.
\end{remark}

Looking at the preceding example, we may ask under what conditions an unramified extension is a tower of only unramified minimal extensions. We begin by characterizing minimal unramified extensions.

\begin{proposition} \label{1.131} Let $R\subset S$ be a minimal extension. Then, $R\subset S$ is unramified if and only if $R\subset S$ is either decomposed or separable inert or minimal Pr\"ufer.
\end{proposition}

\begin{proof} We first observe  that a ramified minimal extension $R\subseteq S$ with conductor $M$ a maximal ideal is luckily not  unramified, otherwise $R/M\to(R/M)[X]/(X^2 )$ would be \'etale 
 and then $(R/M)[X]/(X^2 )$ reduced.

Now a decomposed extension is modulo its conductor ideal of the form $k\subseteq k\times k$, where $k$ is a field, whence is unramified.    

A minimal inert separable extension is unramified, as we said before Remark \ref{1.12}. 

At last, a minimal Pr\"ufer extension is unramified because it is an epimorphism by Proposition \ref{SEP}.
\end{proof}

\begin{proposition} \label{1.132} Let $R\subset S$ be an FCP integral extension. The following statements are equivalent:
\begin{enumerate}

\item $R\subset S$ is a seminormal   $\kappa$-separable  extension;

\item $R\subset S$ is a tower of finitely many decomposed minimal extensions and inert separable   extensions.
\end{enumerate} 

If these conditions hold, then $R\subset S$ is unramified.
\end{proposition}

\begin{proof}  $R\subset S$ is a seminormal extension if and only if 
  $R\subset S$ is a tower of decomposed minimal extensions and inert    extensions by Proposition  \ref{1.31}, which proves a part of the equivalence. 
  
Assume that the previous conditions hold. Let $R_0:=R\subset\ldots\subset R_i\subset\ldots\subset R_n:=S$ be a maximal chain where $R_i\subset R_{i+1}$ is minimal for each $i\in\{0,\ldots,n-1\}$. Let $Q\in\mathrm {Spec}(S)$. Set $P:=Q\cap R$ and $P_i:=Q\cap R_i$ for each $i\in\{0,\ldots,n-1\}$; so that, $P_i=P_{i+1}\cap R_i$. Then, $\kappa_{R_i}(P_i)\to\kappa_{R_{i+1}}(P_{i+1})$ is either an isomorphism or a minimal field extension. In this last case, $\kappa_{R_i}(P_i)\to\kappa_{R_{i+1}}(P_{i+1})$ is either radicial or separable. It follows that $\kappa_R(P)\to \kappa_S(Q)$ is separable if and only if $\kappa_{R_i}(P_i)\to \kappa_{R_{i+1}}(P_{i+1})$ is either an isomorphism or a minimal separable field extension for each $i\in\{0,\ldots,n-1\}$ if and only if $R_i\subset R_{i+1}$ is either a decomposed minimal extension or an inert separable  minimal extension for each $i\in\{0,\ldots,n-1\}$.

The last result comes from Propositions  \ref{1.131} and \ref{1.11}. 
\end{proof}

We will define the unramified closure of an FCP integral extension in Section  6. 
 We are now going to show the existence of the {\it $\kappa$-separable closure} of an FCP integral ring extension $R\subseteq S$ as the greatest subalgebra $V:={}^\square_SR\in[R,S]$ such that $R\subseteq V$ is $\kappa$-separable.

\begin{remark}\label{TRANS} Contrary to the unramified context, we can note that the $\kappa$-separable and $\kappa$-radicial  properties verify, 
   due to Field Theory, and without any additional hypotheses, the following statement. Let $R\subset S\subset T$ be a tower of integral FCP ring extensions, then $R\subset T$ has the $\kappa $-separable (or the $\kappa$-radicial) property if and only if so do $R\subset S$ and $S \subset T$.   
   \end{remark}
 
\begin{lemma} \label{1.19} If $R\subset S$ is an integral FCP t-closed extension, there exists a greatest subalgebra $V:={}^\square_SR\in[R,S]$  such that $R\subseteq V$ is $\kappa$-separable.
\end{lemma}
\begin{proof} We begin to assume that $(R,M)$ is a local ring, and so is $(S,M)$ by \cite[Proposition 4.10]{DPP2}. Since $M=(R:S)$, there is an order preserving  bijection $[R,S]\to[R/M,S/M]$ defined by $T\mapsto T/M$ according to \cite[Lemma II.3]{DMPP}. But $R/M\subset S/M$ is an algebraic field extension, and admits a separable closure $W$ which is the greatest subextension of $R/M\subset S/M$  such that $R/M\subseteq W$ is a separable algebraic field extension. Then, the previous bijection shows that there exists a greatest subalgebra $V:={}^\square_SR\in[R,S]$  such that $R\subseteq V$ is $\kappa$-separable.

Assume now that $R\subset S$ is an arbitrary integral FCP t-closed extension and set $\mathrm{MSupp}(S/R):=\{M_1,\ldots,M_n\}$. Define $\varphi:[R,S]\to\prod_{i=1}^n[R_{M_i},S_{M_i}]$ by $\varphi(T)=(T_{M_1},\ldots,T_{M_n})$. By \cite[Theorem 3.6]{DPP2}, $\varphi$ is a bijection. Let $M_i\in\mathrm{MSupp}(S/R)$. Then $R_{M_i}\subset S_{M_i}$ is an integral FCP t-closed extension, where $(R_{M_i},M_iR_{M_i})$ a local ring. The first part of the proof shows that there exists a greatest subalgebra $V_i\in[R_{M_i},S_{M_i}]$ such that $R_{M_i}\subseteq V_i$ is $\kappa$-separable. Let $V$ be the unique element of $[R,S]$ such that $V_{M_i}=V_i$ for each $i\in\mathbb N_n$. Moreover, $R\subset S$ is an i-extension, being t-closed by Proposition \ref{1.31}. Let $N_i$ be the unique maximal ideal of $S$ lying above $M_i$; so that, $S_{M_i}\cong S_{N_i}$. Then, $\kappa(M_i)=R/M_i\cong R_{M_i}/M_iR_{M_i}$ and $\kappa(N_i)=S/N_i\cong S_{M_i}/M_iR_{M_i}$ because $M_iR_{M_i}=(R_{M_i}:S_{M_i})$. For each $i\in\mathbb N_n$, set $P_i:=N_i\cap V\in\mathrm{Max}(V)$. Then, we have the fields extensions $\kappa(M_i)\subseteq\kappa(P_i)\subseteq\kappa(N_i)$, with $\kappa(P_i)=V/P_i$ the greatest subalgebra of $[\kappa(M_i),\kappa(N_i)]$ such that $\kappa(M_i)\subseteq\kappa(P_i)$ is separable. Then, there exists a $\kappa$-separable closure $V$ of $R\subset S$.
\end{proof}

\begin{theorem} \label{1.20} An integral FCP extension $R\subset S$ admits a greatest subalgebra $V:={}^\square_SR\in[R,S]$ such that $R\subseteq V$ is $\kappa$-separable. Moreover, any subextension $U$ of $S$ is $\kappa$-separable over $R$ if and only if $U\in[R,V]$ and any $Z\in[V,S]$ is such that $Z\subseteq S$ is $\kappa$-radicial. 
\end{theorem}
\begin{proof} Let $T:={}^t_SR$ be the t-closure of $R\subset S$. Since $T\subseteq S$ is t-closed, it admits a $\kappa$-separable closure $V$ by the previous lemma. Let $P\in\mathrm{Max}(R),\ Q\in\mathrm{Max}(V)$ lying over $P$ and $N:=Q\cap T$ which is also lying over $P$. Then, $\kappa(P)\cong\kappa(N)$ because $R\subseteq T$ is infra-integral and $\kappa(N)\subseteq\kappa(Q)$ is separable by definition of $V$. It follows that $\kappa(P)\subseteq\kappa(Q)$ is separable, and $R\subseteq V$ is $\kappa$-separable. Obviously, any subextension of $V$ is $\kappa$-separable over $R$. Let $Z\in[V,S]$ and let $M\in \mathrm{Max}(Z)$ be lying above $Q$, which is also lying over $P$. Then, $\kappa(Q)\subseteq\kappa(M)$ is radicial by definition of $\kappa(Q)$; so that, $Z\subseteq S$ is $\kappa$-radicial.

Let $U\in[R,S]$ be such that $R\subset U$ is $\kappa$-separable. If  $U\in[T,S]$, then $T\subseteq U$ is $\kappa$-separable; so that, $U\subseteq V$. Assume that $U\not\in[T,S]$. We claim that $U\in [R,V]$. Set $W:=U\cap T\subset T$, which is the t-closure of $R\subset U$. We have the following commutative diagram:
\centerline{$\begin{matrix}
 R   & \to &       W           & \to &         U         &  {}  & {} \\
 {}   &  {}  & \downarrow &  {}  & \downarrow &  {}  & {} \\
 {}   & {}  &         T           & \to &       TU         & \to & S
\end{matrix}$} 
Since $W\subset T$ is infra-integral, there exists a directed set $\{W_i\}_{i\in I}$ of $W$-subalgebras of $T$ such that $W\subset W_i$ is ct-elementary and $T=\cup_{i\in I}W_i$ by \cite[Definition 1.2]{Pic 0}. It  obviously follows  that  there exists a directed set $\{U_i\}_{i\in I}$ of $U$-subalgebras of $TU$ such that $U\subset U_i$ is ct-elementary and $TU=\cup_{i\in I}U_i$. Indeed, if $A\subset B$ is a t-elementary subextension of $[W,T]$,   $UA\subset UB$  is a t-elementary subextension of $[U,TU]$. Therefore, $U\subseteq TU$ is infra-integral, and then $\kappa$-separable by an above  remark. 
 Hence $R\subset UT$ is $\kappa$-separable, and so is $T\subseteq UT$; so that, $U\subseteq UT\subseteq V$, and $U\in[R,V]$. 
\end{proof}

\begin{corollary} \label{1.21} Let $R\subset S$ be an integral FCP extension. Then, ${}^\square_SR\subseteq S$ is radicial and ${}^t_SR={}^\square_SR\cap{}^\circ_SR$. 
\end{corollary}
\begin{proof} Set $T:={}^t_SR,\ V:={}^\square_SR$ and $U:={}^\circ_SR$. Let $Q\in\mathrm{Max}(S),\ P:=Q\cap T$ and $M:=Q\cap R=P\cap R$. Since $R\subseteq T$ is infra-integral, we have $\kappa(M)\cong\kappa(P)$. But we also have $V={}^\square_ST$ by the proof of Theorem \ref{1.20}. Hence, $N:=Q\cap V$ is the unique maximal ideal of $V$ lying over $P$. Moreover, $V_P/PT_P=\kappa(N)$ is the separable closure of the field extension $\kappa(P)\subseteq\kappa(Q)$ as a consequence of the construction made in Lemma \ref{1.19}; so that, $\kappa(N)\subseteq\kappa(Q)$ is radicial. At last, $\mathrm{Spec}(S)\to\mathrm{Spec}(V)$ is bijective since $V\subseteq S$ is t-closed by Proposition \ref{1.31}, whence $V\subseteq S$ is radicial.

In particular, $T\subseteq U\cap V$ since the residual extensions of $R\subseteq T$ are isomorphisms. Assume that $T\subset U\cap V$. Then, the residual extensions of $T\subset U\cap V$ are both radicial and separable fields extensions, that is isomorphisms; so that, $T\subset U\cap V$ is infra-integral, a contradiction. To conclude, $T= U\cap V$.
\end{proof}

\section { u-closures of  integral FCP extensions}

In this section, we work in the context of FCP integral extensions, where 
 many results of \cite{Pic 0} about anodality  will be easier  to handle.

\begin{proposition}\label{1.016} An integral FCP extension $R\subseteq S$ is u-closed if and only if $R/I\subseteq S/I$ is u-closed for any ideal $I$ shared by $R$ and $S$.
\end{proposition}

\begin{proof} Obvious.
\end{proof}

The next result does not hold if $ R\subset S$ is not FCP.

\begin{proposition}\label{1.311} Let $R\subset S$ be an integral FCP extension. 
Then the following conditions are equivalent:
\begin{enumerate}
\item $R\subset S$ is u-closed;

\item  $R\subset S$ is an i-extension;

\item  ${}_S^+R= {}_S^tR$. 
\end{enumerate}

If these conditions are satisfied, any minimal subextension of $R\subset S$ is either ramified or inert.  
\end{proposition} 

\begin{proof} 
(1) $\Rightarrow$ (2) Set $I:=(R:S)$. Proposition \ref{1.016} shows that $R/I\subseteq S/I$ is u-closed. Moreover, $R/I$ is an Artinian, and so a Noetherian ring by \cite[Theorem 4.2]{DPP2}. Then, by \cite[Proposition 2.35]{Pic 0}, $\mathrm{Spec}(S/I)\to\mathrm{Spec}(R/I)$ is bijective. It follows that $\mathrm{Spec}(S)\to\mathrm{Spec}(R)$ is bijective.

(2) $\Rightarrow$ (1) Assume that $\mathrm{Spec}(S)\to\mathrm{Spec}(R)$ is bijective. Then \cite[Theorem 2.27]{Pic 0} asserts that $R\subset S$ is u-closed.

(2) $\Rightarrow$ (3) ${}_S^+R\neq{}_S^tR$ implies that $\mathrm{Spec}({}_S^tR)\to\mathrm{Spec}({}_S^+R)$ is not bijective according to Theorem \ref{minimal} since there is a minimal decomposed subextension of ${}_S^+R\subset{}_S^tR$.

(3) $\Rightarrow$ (2) If ${}_S^+R={}_S^tR$, then $\mathrm{Spec}(S)\to\mathrm{Spec}(R)$ is bijective since so are $\mathrm{Spec}(S)\to\mathrm{Spec}({}_S^tR)$ and $\mathrm{Spec}({}_S^+R)\to\mathrm{Spec}(R)$ by Proposition~\ref{1.31}.

The last assertion is an easy consequence of (3). 
\end{proof}

\begin{corollary}\label{1.315} A minimal extension $R\subset S$ is u-closed if and only if $R\subset S$ is either ramified or inert or minimal Pr\"ufer.  
\end{corollary} 

\begin{proof} The equivalence in the integral case is a consequence of Proposition ~\ref{1.311} since $\mathrm{Spec}(S)\to\mathrm{Spec}(R)$ is bijective if and only if $R\subset S$ is either ramified or inert. Assume that $R\subset S$ is a minimal Pr\"ufer extension. Then $R\subset S$ is u-closed since t-closed by Proposition ~\ref{1.34}.
\end{proof}

\begin{proposition}\label{1.017} Let $R\subseteq S$ be a u-closed integral FCP extension. Then $T\subseteq U$ is u-closed for any $T,U\in[R,S]$ such that $T\subseteq U$.
\end{proposition}
\begin{proof} We use Proposition \ref{1.311}. Assume that $\mathrm{Spec}(U)\to\mathrm{Spec}(T)$ is not bijective. There exist $N_1,N_2\in\mathrm{Max}(U),\ N_1\neq N_2$, such that $N_1\cap T=N_2\cap T$. In particular, $N_1\cap R=N_2\cap R$. Since, $R\subseteq S$ is an FCP integral extension, so are $T\subset U$  and $U\subset S$. In particular, there exist $M_1,M_2\in\mathrm{Max}(S)$ such that $N_1=M_1\cap U$ and $N_2=M_2\cap U$. It follows that $M_1\cap R=M_2\cap R$ with $M_1\neq M_2$ since $N_1\neq N_2$, a contradiction with $\mathrm{Spec}(S)\to\mathrm{Spec}(R)$ bijective. Then, $\mathrm{Spec}(U)\to\mathrm{Spec}(T)$ is bijective and $T\subseteq U$ is u-closed by Proposition \ref{1.311}.
\end{proof} 

\begin{proposition}\label{1.312} Let $R\subset S$ be an integral FCP extension. 
\begin{enumerate}
\item $({}_S^uR)_M={}_{S_M}^uR_M$ for any $M\in\mathrm{MSupp}(S/R)$.

\item $({}_S^uR)/I={}_{(S/I)}^u(R/I)$ for any ideal $I$ shared by $R$ and $S$.
\end{enumerate}
\end{proposition} 

\begin{proof} Set $U:={}_S^uR$.

(1) Let $M\in\mathrm{MSupp}(S/R)$. We first show  that $U_M\subseteq S_M$ is u-closed. Using Proposition \ref{1.311}, it is enough to show that $\mathrm{Spec}(S_M)\to\mathrm{Spec}(U_M)$ is bijective, which is obvious since so is $\mathrm{Spec}(S)\to\mathrm{Spec}(U)$ because $U\subseteq S$ is u-closed. Then, ${}_{S_M}^uR_M\subseteq U_M\ (*)$. 
 
Now, since $R\subseteq S$ is an integral FCP extension, \cite[Theorem 3.6]{DPP2} shows that there exists $U'\in [R,S]$ such that $U'_M={}_{S_M}^uR_M$ for any $M\in\mathrm{MSupp}(S/R)$. In particular, $U'_M\subseteq S_M$ is u-closed for any $M\in\mathrm{MSupp}(S/R)$. Therefore, $U'\subseteq S$ is u-closed by \cite[Proposition 1.16]{Pic 0}. Then, $U\subseteq U'$, which implies $U_M\subseteq U'_M={}_{S_M}^uR_M$ for any $M\in\mathrm{MSupp}(S/R)$; so that,  $U_M=({}_S^uR)_M={}_{S_M}^uR_M$ for any $M\in\mathrm{MSupp}(S/R)$ by $(*)$. 

(2) Let $I$ be an ideal shared by $R$ and $S$. Then $I$ is also an ideal of $U$. Since $U\subseteq S$ is u-closed, so is $U/I\subseteq S/I$ by Proposition \ref{1.016}; so that, ${}_{(S/I)}^u(R/I)\subseteq U/I$. But $\mathrm{Spec}(S/I)\to\mathrm{Spec}({}_{(S/I)}^u(R/I))$ is bijective, with ${}_{(S/I)}^u(R/I)\in [R/I,S/I]$. It follows that there exists $U''\in[R,S]$ such that $U''/I={}_{(S/I)}^u(R/I)$, with $U''/I\subseteq U/I$, giving $U''\subseteq U\ (**)$. Moreover, $\mathrm{Spec}(S/I)\to\mathrm{Spec}(U''/I)$ is bijective, and so is $\mathrm{Spec}(S)\to\mathrm{Spec}(U'')$. At last, Proposition \ref{1.311} shows that $U''\subset S$ is u-closed, which implies $U\subseteq U''$, and, to conclude, $U=U''$ by $(**)$, that is $U/I={}_{(S/I)}^u(R/I)$.
\end{proof} 

\begin{corollary}\label{1.3120} Let $R\subset S$ be an integral FCP extension. The following conditions are equivalent
\begin{enumerate}
\item $R\subset S$ is u-closed (resp.; u-integral);

\item $R_M\subset S_M$ is u-closed (resp.; u-integral) for any $M\in\mathrm{MSupp}(S/R)$;

\item $R/I\subset S/I$ is u-closed (resp.; u-integral) for any ideal $I$ shared by $R$ and $S$.
\end{enumerate}
\end{corollary} 

\begin{proof} We use  Proposition \ref{1.312}. Then $({}_S^uR)_M={}_{S_M}^uR_M$ for any $M\in\mathrm{MSupp}(S/R)$ and $({}_S^uR)/I={}_{(S/I)}^u(R/I)$ for any ideal $I$ shared by $R$ and $S$. The different equivalences are obvious since $R\subset S$ is u-closed (resp.; u-integral) if and only if ${}_S^uR=R$ (resp.; ${}_S^uR=S$), with similar equalities for $R_M\subset S_M$ and $R/I\subset S/I$.
\end{proof} 

We recall \cite[Definition 2.11]{Pic 0} that a ring morphism $f:A\to B$ is said to be {\it locally an epimorphism} if $A_{f^{-1}(Q)}\to B_Q$ is an epimorphism for every $Q\in\mathrm{Spec}(B)$.

The next result gives a characterization of $u$-integral extensions in the FCP context  by some property of ring morphisms as it is the case for subintegral and infra-integral extensions.  In an arbitrary context this result may not hold.

\begin{theorem}\label{1.320} An integral FCP extension $R\subset S$ is u-integral if and only if $R\subset S$ is locally an epimorphism, in which   case $R\subset S$ is unramified.
\end{theorem} 

\begin{proof} Assume first that $R\subset S$ is u-integral. Then, $R\subset S$ is locally an epimorphism by \cite[Corollary 2.23]{Pic 0}. Conversely, assume that $R\subset S$ is locally an epimorphism. Since 
${}_S^uR\subseteq S$ is u-closed, Proposition \ref{1.311} asserts that $\mathrm {Spec}(S)\to\mathrm{Spec}({}_S^uR)$ is bijective. Moreover, $R\subset S$ is a finite injective locally epimorphism. According to \cite[Remark 2.34 (1)]{Pic 0},  $R\subset S$ is cu-elementary, and, in particular, u-integral. The last statement holds by Proposition \ref{RAM}(2) 
 and \cite[Exercise 1.14, p.7]{I}.
\end{proof} 

The seminormalization, the t-closure and the u-closure are linked as it is shown in the following Theorem. 

\begin{theorem}\label{1.313} If $R\subset S$ is an integral FCP extension, then $({}_S^uR)({}_S^+R)={}_S^tR$ holds. In particular, such an extension is t-closed if and only if it is seminormal and u-closed (anodal).
\end{theorem} 
\begin{proof} 
Since $\mathrm{Spec}(S)\to\mathrm{Spec}({}_S^tR)$ is bijective, Proposition \ref{1.311} shows that ${}_S^tR\subset S$ is u-closed; so that, ${}_S^uR\subseteq{}_S^tR$. 
We already know that ${}_S^+R\subseteq{}_S^tR$ by Definition \ref{1.3}, whence  $({}_S^uR)({}_S^+R)\subseteq{}_S^tR$ holds. 

Now ${}_S^uR{}\subseteq({}_S^uR)({}_S^+R)$ and ${}_S^+R\subseteq({}_S^uR)({}_S^+R)$ combine to yield that $({}_S^uR)({}_S^+R)\subseteq S$ is both seminormal and u-closed, and then locally u-closed by Corollary \ref{1.3120}. Hence, $({}_S^uR)({}_S^+R)\subseteq S$ is t-closed, according to \cite[Proposition 3.17]{Pic 1}, leading to ${}_S^tR\subseteq({}_S^uR)({}_S^+R)$, and then to $({}_S^uR)({}_S^+R)={}_S^tR$.

The last statement is \cite[Proposition 3.17]{Pic 1} since u-closedness is equivalent to local u-closedness by Corollary \ref{1.3120}.
\end{proof} 

\begin{theorem}\label{1.319} Let $R\subset S$ be an integral FCP extension. The following conditions are equivalent: 
\begin{enumerate}
\item $R\subset S$ is  u-integral  and seminormal;

\item $R\subset S$ is seminormal infra-integral;

\item $|\mathrm{V}_S((R:S))|=|\mathrm{V}_R((R:S))|+\ell[R,S]$;

\item ${}_S^+R=R$ and ${}_S^tR=S$.

\item $(R:S)$ is an irredundant intersection of $\ell[R,S]+n$ maximal ideals of $S$ and an intersection of $n$ maximal ideals of $R$ for some positive integer $n$. 
\end{enumerate}
 If these conditions hold, then $n=|\mathrm{V}_R((R:S))|$.
\end{theorem} 

\begin{proof} 
 Set $n:=|\mathrm{V}_R((R:S))|<\infty$ by Proposition \ref{SUP}.

(1) $\Rightarrow$ (2) Since $R\subset S$ is u-integral, then $R\subset S$ is infra-integral by Definition \ref{1.33} and Proposition \ref{1.36}.

(2) $\Rightarrow$ (1) Since $R\subset S$ is infra-integral, we have $S={}_S^tR$ and since $R\subset S$ is seminormal, we have $R={}_S^+R$. Then, Theorem \ref{1.313} leads to ${}_S^uR=({}_S^uR)R=({}_S^uR)({}_S^+R)={}_S^tR=S$; so that, $R\subset S$ is  u-integral.

(2) $\Leftrightarrow$ (3) because of Proposition  \ref{1.31}.

(2) $\Leftrightarrow$ (4) is obvious according to the definitions of ${}_S^+R$ and ${}_S^tR$.

(2) $\Rightarrow$ (5) Since $R\subset S$ is seminormal FCP, $(R:S)=\cap_{i=1}^mM_i$, where $M_i\in\mathrm{Max}(S)$ for each $i\in\mathbb N_m$ by Proposition \ref{1.91}; so that, $m=|\mathrm{V}_S((R:S))| =|\mathrm{V}_R((R:S))|+\ell[R,S]$ by (3). But, we also have $(R:S)=\cap_{j=1}^nP_j$, where the $P_j$'s are the intersections of the $M_i$'s with $R$,  with perhaps  more than one $M_i$ for one $P_j$. Then,  $m=n+\ell[R,S]$, giving (5). 

(5) $\Rightarrow$ (3) Assume that $(R:S)=\cap_{i=1}^mM_i$, where $M_i\in\mathrm{Max}(S)$ for each $i\in\mathbb N_m$, with $m=\ell[R,S]+n$ and that $(R:S)=\cap_{j=1}^nP_j$, where the $P_j\in\mathrm{Max}(R)$ for each $j\in\mathbb N_n$. Then, $|\mathrm{V}_S((R:S))|=|\mathrm{V}_R ((R:S))|+\ell[R,S]$, giving (3).
\end{proof} 

The following Proposition was proved in \cite[Theorem 4.8]{Pic 0} in a different context.

\begin{proposition}\label{1.316} \cite[Theorem 4.8]{Pic 0} A minimal extension $R\subset S$ is u-integral if and only if $R\subset S$ is  decomposed. 
\end{proposition} 
\begin{proof} 
 Let $R\subset S$ be a minimal extension and set $M:=\mathcal{C}(R,S)$.

If $R\subset S$ is decomposed, there exists $q\in S$ such that $S=R[q],\ q^2-q\in M\subset R$ and $Mq\subseteq M$; so that, $q^3-q^2\in R$. Then,  $R\subset S$   is u-integral.

Conversely, assume that $R\subset S$ is u-integral. 
By assumption, $R\subset S$ is a minimal integral extension. If $R\subset S$ is either minimal ramified or minimal inert, then $R\subset S$ is an i-extension; so that, $R\subset S$ is u-closed, which implies $R=S$, a contradiction. Then, $R\subset S$ is  decomposed.
\end{proof} 
\begin{remark}\label{1.317} Let $R\subset S$ be an integral FCP extension with ${}_S^+R,{}_S^tR\neq R,S$. 
 According to Proposition \ref{1.31}, a minimal subextension of $R\subset{}_S^+R$ is ramified and a minimal subextension of ${}_S^+R\subset S$ is either decomposed or inert. 
 Similarly, any minimal subextension of $R\subset{}_S^tR$ is either ramified or decomposed and any minimal subextension of ${}_S^tR\subset S$ is inert. Contrary to these cases, we have seen in Proposition \ref{1.311} that any minimal subextension of an u-closed extension is either ramified or inert. But a minimal ramified extension can also be a subextension of an u-integral extension as we can see in the following examples. We have not the same behavior with the u-closure where we may have two minimal extensions of the same type contained and containing the u-closure. Let $R\subset S$ be an integral FCP extension with u-closure ${}_S^uR$. We can see in the following examples that there exists a minimal ramified subextension of $R\subset{}_S^uR$ and a minimal ramified subextension of ${}_S^uR\subset S$

(1) Let $R$ be an Artinian non reduced ring. Then, $R\subset R^2$ is an FCP infra-integral and not seminormal extension by \cite[Proposition 1.4]{Pic 9}. It follows that there exists $U\in[R,R^2]$ such that $R\subset U$ is minimal ramified. Moreover, $R\subset R^2$ is u-integral by \cite[Example 2.32]{Pic 0}. This example also shows that a subextension of an u-integral extension is not necessarily u-integral. Indeed, $R\subset{}_{R^2}^+R$ is u-closed since subintegral by Proposition \ref{1.311}. 

 (2) We continue to work  with the situation of (1),
with $R:=$

\noindent $(\mathbb Z/2\mathbb Z)[T]/(T^2)=(\mathbb Z/2\mathbb Z)[t]$, where $t$ is the class of $T$ in $R$. Then, $R$ is an Artinian local non reduced ring. Let $M:=Rt$ be its maximal ideal. It follows that $R\subset R^2$ is u-integral. According to  \cite[Proposition 2.8]{Pic 9}, $R+(M\times M)={}_{R^2}^+R$, with $R\subset R+(M\times M)$ minimal ramified. Indeed, setting $e_1:=(1,0)$, we have $R+(M\times M)=R[te_1]$, with $(te_1)^2=0\in R$ and $Mte_1=Rt^2e_1=0$. And $R+(M\times M)\subset R^2$ is minimal decomposed since $|\mathrm{Max}(R^2)|=2$. Then, $\ell [R,R^2]=2$. Set $S:=R\times R[x]$, where $R[x]:=R[X]/(X^2,tX)$ and $x$ is the class of $X$ in $R[X]/(X^2,tX)$. Since $x^2=0$ and $Mx=Rtx=0$, we get that $R\subset R[x]$ is minimal ramified, as $R^2\subset S$, by \cite[Proposition III.4]{DMPP}; so that, $\ell[R,S]=3$. We claim that $R^2= {}_S^uR$. Set $U:={}_S^uR$. Then, $U\neq R$ since $R\subset S$ is infra-integral and not subintegral, because $|\mathrm{Max}(R^2)|=2$. For a similar reason, $U\neq S$ because $R^2\subset S$ is subintegral. Then, $R\subset U\subseteq R^2\subset S$, which implies that $\ell[R,U]\in\{1,2\}$. If $\ell[R,U]=2$, then $U=R^2$, and we are done. Assume that $\ell[R,U]=1$; so that, $R\subset U$ is minimal, necessarily decomposed, since u-integral and $U\subset R^2$ is minimal ramified. Considering the extension $R\subset R^2$ and using \cite[Lemma 17]{DPP4}, we get that $MR^2=M\times M$ is not a radical ideal of $R^2$, a contradiction. Then, $\ell[R,U]= 2$ and $U=R^2={}_S^uR$. This example exhibits a subintegral minimal ramified subextension $R\subset{}_{R^2}^+R$ of $R\subset{}_S^uR=R^2 $ as well as a  subintegral minimal ramified subextension ${}_S^uR\subset S$ of (itself) ${}_S^uR\subset S$ as we can see in the following commutative diagram:
\centerline{$\begin{matrix}
  {}  & {}   &        {}           &       {}       &   {}_S^+R       &      {}       & {} \\
  {}  & {}   &        {}           & \nearrow &         {}            & \searrow & {} \\
 R   & \to & {}_{R^2}^+R &      {}        &         {}            &       {}      & S \\
 {}   & {}  &        {}            & \searrow &         {}            & \nearrow & {} \\
 {}   & {}  &        {}            &      {}       & {}_S^uR=R^2 &       {}       & {}
\end{matrix}$} 
\end{remark} 

In the preceding Remark we can observe that $R\subset{}_S^+R\cap{}_S^uR={}_{R^2}^+R$. It may be asked under what conditions $R={}_S^+R\cap{}_S^uR$ holds for an FCP extension $R\subset S$. The answer is given in the following Proposition.

\begin{proposition}\label{1.318} Let $R\subset S$ be an  FCP extension. Then, $R={}_S^+R\cap{}_S^uR$ if and only if $R\subseteq {}_S^uR$ is seminormal.
\end{proposition} 
\begin{proof} Set $U:={}_S^uR$. Then, $T:={}_S^+R\cap U={}_U^+R$. It follows that $T=R$ if and only if $R={}_U^+R$ if and only if $R\subseteq U$ is seminormal.
\end{proof} 

The {\it $\omega$-closure} of a ring extension $R\subseteq S$ is the greatest element $U:={}^\omega_SR$ of $[R,S]$ such that $R\subseteq U$ is unramified \cite[Proposition 3.3]{Pic 0}.

\begin{proposition} \label{1.14} Let $R\subseteq S$ be an FCP integral extension. There exists a greatest unramified subextension $R\subseteq  {}^\omega_SR$ of $R\subseteq S$. Then   ${}^\omega_SR\subseteq S$ is $u$-closed; so that, 
 ${}^u_SR\subseteq{}^\omega_SR\subseteq{}^\square_SR$.   
\end{proposition}

\begin{proof} For the first part use \cite[Proposition 3.3]{Pic 0},  taking into account that all subextensions of an FCP extension are of finite type. 
To show that ${}^\omega _SR \subseteq S$ is $u$-closed, consider some $s\in S$, such that $s^2-s, s^3-s^2 \in {}^\omega _SR$, then ${}^\omega _SR\subseteq{}^\omega _SR[s]$ is $u$-integral, whence unramified according to Theorem ~\ref{1.320}, and so is $R\subseteq {}^\omega _SR[s]$. By maximality of the unramified closure, $s$ belongs to ${}^\omega _SR$.  Since ${}^\omega _SR \subseteq S$ is $u$-closed, ${}^u _SR \subseteq {}^\omega _SR$. 
 Now ${}^\omega_SR\subseteq{}^\square_S R$ holds by Proposition \ref{RAM} since $R\subseteq{}^\omega_SR$ is unramified.   
\end{proof}

\begin{corollary} \label{1.16} Let $R\subset S$ be a $\kappa$-radicial FCP extension (for instance, infra-integral). There exists some $\Sigma\in[R,S]$ such that $R\subseteq\Sigma$ is unramified, $\Sigma\subseteq S$ is radicial and $\Sigma = {}^\omega _SR={}^u_SR$. 
\end{corollary}
\begin{proof} It is enough to observe that the extension ${}^u_SR \subseteq{}^\omega_SR$ is unramified because so is $R\subseteq{}^\omega_SR$, and radicial because $\kappa$-radicial and an $i$-extension by Proposition \ref{1.311}. Moreover, a radicial unramified FCP extension is trivial by Proposition \ref{SEP}. Then, set $\Sigma:={}^\omega _SR={}^u_SR$. 
\end{proof}

\begin{remark} \label{1.17}  Theorem \ref{1.20} shows that for an integral FCP extension $R\subset S$, the extension $R\subseteq{}^\square_SR$ is $\kappa$-separable and ${}^\square_SR\subseteq S$ is $\kappa$-radicial.  
The following example shows that we cannot generalize the above Corollary to arbitrary integral extensions, which means that we cannot replace the $\kappa$-separable property by the unramified property. 
 We use  \cite[Example 4.10]{Pic 6} to get such a situation.

Let $K\subset L$ be a minimal separable field extension of degree 2; so that, there exists $y\in L$ such that $L=K[y]$. Set $S:=L[X]/(X^2)$ and let $x$ be the class of $X$ in $S$. Set $R:=K+Kx$ and $T:=K+Lx$; so that, 
$R$ is a local ring with maximal ideal $M:=Kx$ and $T$ is a local ring with maximal ideal $M':=M+Rxy$. It was proved in \cite[Example 4.10]{Pic 6} that $R\subset T$ is minimal ramified, $T\subset S$ is minimal inert with $M'=(T:S)$ and $T/M'\subset S/M'$ is a minimal separable field extension with $|[R,S]|=3$. Moreover, $R\subset T$ is radicial because an i-extension with residual extensions which are isomorphisms and $T\subset S$ is unramified by Proposition ~\ref{RAM}. Since $|[R,S]|=3$, there does not exist $\Sigma\in[R,S]$ such that $R\subset\Sigma$ is unramified and $\Sigma\subset S$ is radicial. 
\end{remark}

\section{Co-subintegral closure and co-infra-integral closure of an integral FCP extension}

Definition \ref{closure} recalls the definitions of several closures. 
 Let $R\subset S$ be a ring extension. As we did in Section 3 with the co-integral closure, we are going to characterize some closures defined as the least $T\in [R,S]$ such that $T\subseteq S$ satisfies a given property.
 
Let $R\subseteq S$ be an integral FCP extension. In \cite[Lemma 17]{DPP4}, we give a condition under which there exists $T\in[R,S]$ such that $T\subset S$ is  minimal ramified. 

A more general requirement is as follows: Let $R\subseteq S$ be an  FCP extension.
 There exists some $U\in[R,S]$ which is the least $V\in[R,S]$ such that $V\subseteq S$ is a subintegral extension. If such a ring  $U$ exists, we call it the {\it co-subintegral closure} of $R\subseteq S$ and denote it by $S^+_R$ since it is a kind of dual of the seminormalization which is nothing but the subintegral closure ${}^+_SR$  and the largest  $V\in[R,S]$ such that $R\subseteq V$ is subintegral. 

Similarly, we can ask if there exists some $U\in[R,S]$ which is the least $V\in[R,S]$ such that $V\subseteq S$ is an infra-integral extension. If this $U$ exists, we call it the {\it co-infra-integral closure} of $R\subseteq S$ and denote it by $S^t_R$ since it is a kind of dual of the t-closure ${}^t_SR$ (the largest $V\in[R,S]$ such that $R\subseteq V$ is infra-integral). Many of the proofs of this section are similar for the subintegral case and the infra-integral case.

If these co-subintegral closure or co-infra-integral closures exist, we have $\underline R\subseteq S^t_R\subseteq S^+_R\subseteq S$, with $\underline R\subseteq S$ an integral FCP extension. That is the reason why we limit ourselves to consider only integral FCP extensions.

\begin{proposition}\label{1.6} Let $R\subset S$ be an integral FCP extension. The co-subintegral (resp.; co-infra-integral) closure exists in the following cases:
\begin{enumerate}
\item $R\subset S$ is subintegral (resp.; infra-integral). Then, $R=S^+_R$ (resp.; $=S^t_R$).

\item $R\subset S$ is seminormal (resp.; t-closed). Then, $S=S^+_R$ (resp.; $=S^t_R$) and ${}_S^uR={}_S^tR$.
\end{enumerate}
\end{proposition} 
\begin{proof} Obvious.
\end{proof} 

\begin{corollary}\label{1.81} Let $R\subset S$ be an integral FCP extension. Then, $R=({}_S^+R)^+_R=({}_S^+R)^t_R=({}_S^tR)^t_R$. 
\end{corollary} 
\begin{proof} Since $R\subseteq{}_S^+R$ (resp.; ${}_S^tR$) is subintegral (resp.; infra-integral), whence infra-integral, Proposition \ref{1.6} gives the answer in both cases.
\end{proof} 

The following remark shows that such  closures do not always exist.

\begin{remark}\label{1.421} Let $k\subset K$ be a minimal radicial field extension of degree 2; so that, $\mathrm{c}(k)=\mathrm{c}(K)=2$. Let $\alpha\in K$ be such that $K=k[\alpha]$, with $a:=\alpha^2\in k$. Set $S:=K[X]/(X^2)=K[x]=K+Kx$, where $x$ is the class of $X$ in $S$. Then, $K\subset S$ is a minimal ramified extension (and then subintegral and infra-integral) with $x^2=0$. Set $\beta:=\alpha+x\in S$ and $L:=k[\beta]\subseteq S$. It follows that $\beta^2=\alpha^2=a\in k$; so that, $k\subset L$ is a minimal radicial field extension with $L\cong k[T]/(T^2-a)\cong K$. Of course $K\neq L$ and $S=L[x]=L+Lx$, with $L\subset S$ a minimal ramified extension  (and then subintegral and infra-integral).  
As $K\cap L=k$, the extension $k\subset S$ has neither a co-subintegral closure nor a co-infra-integral closure, since there is no $T\in[k,S]$ such that $T\subset S$ is subintegral (resp.; infra-integral) with $T\subseteq K,L$.
\end{remark}

\begin{proposition}\label{1.426} Let $R\subset S$ be an integral FCP extension. 
\begin{enumerate}
\item If the co-subintegral (resp.; co-infra-integral) closure $T$ of $R\subset S$ exists, then $T_M$ is the co-subintegral (resp.; co-infra-integral) closure of $R_M\subset S_M$ for any $M\in \mathrm{MSupp}(S/R)$.

\item Assume that for any $M\in\mathrm{MSupp}(S/R)$, there exists a co-subintegral (resp.; co-infra-integral) closure $T_{(M)}$ of $R_M\subset S_M$. Then, the co-subintegral closure (resp.; co-infra-integral) $T$ of $R\subset S$ exists and satisfies $T_M=T_{(M)}$  for any $M\in \mathrm{MSupp}(S/R)$.
\end{enumerate}
\end{proposition} 
\begin{proof} (1) Assume that the co-subintegral (resp.; co-infra-integral) closure $T$ of $R\subset S$ exists; so that, $T$ is the least $U\in[R,S]$ such that $U\subseteq S$ is subintegral (resp.; infra-integral). Then $T_M\subseteq S_M$ is subintegral (resp.; infra-integral) for any $M\in \mathrm{MSupp}(S/R)$ by \cite[Corollary 2.10]{S} (resp.; \cite[ Proposition 1.16]{Pic 1}). Let $V\in[R_M,S_M]$ be such that $V\subseteq S_M$ is subintegral (resp.; infra-integral) for some $M\in \mathrm{MSupp}(S/R)$. Since $R\subset S$ is an integral FCP extension, there exists $W\in[R,S]$ such that $W_M=V$ and $W_{M'}=S_{M'}$ for any $M'\in \mathrm{MSupp}(S/R)\setminus\{M\}$ by \cite[Theorem 3.6]{DPP2}. Then, again by \cite[Proposition 2.9]{S} (resp.; \cite[Proposition 3.6]{Pic 1}), we get that $W\subseteq S$ is  subintegral (resp.; infra-integral); so that, $T\subseteq W$, which leads to $T_M\subseteq W_M=V$. Then $T_M$ is the least $U\in[R_M,S_M]$ such that $U\subseteq S_M$ is subintegral (resp.; infra-integral), which means that $T_M$ is the co-subintegral (resp.; co-infra-integral) closure  of $R_M\subset S_M$.

(2) Conversely, assume that for any $M\in\mathrm{MSupp}(S/R)$, there exists a co-subintegral (resp.; co-infra-integral) closure $T_{(M)}$ of $R_M\subset S_M$; so that, $T_{(M)}\subseteq S_M$ is subintegral (resp.; infra-integral) for any $M\in\mathrm{MSupp}(S/R)$. We deduce from the above reference \cite[Theorem 3.6]{DPP2} that  there exists $T\in[R,S]$ such that $T_M=T_{(M)}$ for any $M\in \mathrm{MSupp}(S/R)$; so that, $T_M\subseteq S_M$ is subintegral (resp.; infra-integral) for any $M\in\mathrm{MSupp}(S/R)$, giving that $T\subseteq S$ is subintegral (resp.; infra-integral) by \cite[Proposition 2.9]{S} (resp.; \cite[Proposition 3.6]{Pic 1}). Now, let $V\in[R,S]$ be such that $V\subseteq S$ is subintegral (resp.; infra-integral). Then, $V_M\subseteq S_M$ is subintegral (resp.; infra-integral) for any $M\in\mathrm{MSupp}(S/R)$. It follows that $T_M=T_{(M)}\subseteq V_M$ for any $M\in\mathrm{MSupp}(S/R)$. To end $T\subseteq V$ shows that $T$ is the co-subintegral (resp.; co-infra-integral) closure of $R\subseteq S$. 
\end{proof} 

 Let $R\subset S$ be an FCP extension.
We recall that $R$ is called {\it unbranched} in $S$ if $\overline R$ is local, in which case all rings in $[R,S]$ are local by \cite[Lemma 3.29]{Pic 11}. We also say that $R\subset S$ is unbranched. A {\it locally unbranched} extension is an extension $R\subset S$ such that $R_P\subset S_P$ is unbranched for each $P\in\mathrm{Spec}(R)$. A locally unbranched integral extension is nothing but an integral $i$-extension.
 
\begin{proposition}\label{1.449} If $R\subset S$ is an integral locally unbranched FCP extension such that the co-subintegral closure  exists,  then $S^+_R= S^t_R$. 
\end{proposition} 
\begin{proof} Since $R\subset S$ is locally unbranched, so is $S^t_R\subseteq S^+_R$; so that, $S^t_R\subseteq S$ is subintegral, and then $S^t_R= S^+_R$.
\end{proof} 

\begin{theorem}\label{1.423} An infra-integral FCP extension $R\subset S$ has a  co-subintegral closure and  ${}_S^uR= S^+_R$. 
\end{theorem} 
\begin{proof} Since $R\subset S$ is infra-integral, ${}_S^uR\subseteq S$ is infra-integral and u-closed, and then an i-extension, by Proposition \ref{1.311}; so that, ${}_S^uR\subseteq S$ is subintegral. In particular, ${}_S^uR$ is the least element $T\in[R,S]$ such that, $T\subseteq S$ is subintegral since the least element $T\in[R,S]$ such that $T\subseteq S$ is an i-extension.
 Then, ${}_S^uR= S^+_R$.
\end{proof} 

\begin{corollary}\label{1.424} If $R\subset S$ is integral and FCP, then $({}_S^tR)^+_R= {}^u_SR$.
\end{corollary} 
\begin{proof} Set $T:=({}_S^tR)^+_R$, which exists by Theorem \ref{1.423} and $U:={}^u_SR$. Since $T\subseteq{}_S^tR$ is subintegral, it is an i-extension, and so is ${}_S^tR\subseteq S$, since t-closed. Then, $T\subseteq S$ is an i-extension, and then u-closed by Proposition \ref{1.311}; so that, $U\subseteq T\subseteq{}_S^tR$. Therefore, $U\subseteq T$ is infra-integral and  an i-extension, whence  subintegral; so that, $U=T$. 
\end{proof} 

\begin{corollary}\label{1.435} Let $R\subset S$ be an integral FCP extension whose co-infra-integral closure exists. Then  the co-subintegral closure of $R\subset S$ exists and   $S^+_R=S^+_{(S^t_R)}$.
\end{corollary} 
\begin{proof} Set $T:=S^t_R$. Then $T\subseteq S$ is infra-integral and admits a co-subintegral closure $U:=S^+_T$ according to Theorem \ref{1.423}. Let $V\in[R,S]$ be such that $V\subseteq S$ is subintegral, whence infra-integral. It follows that $U=S^+_T\subseteq V$. Then $U$ is also the co-subintegral closure of $R\subseteq S$. 
\end{proof} 

\begin{proposition} \label{1.18} An integral FCP extension  $R\subset S$,   which has a co-subintegral closure distinct from $S$ is not unramified.
\end{proposition}
\begin{proof} Assume that $R\subset S$ is unramified and admits a co-subintegral closure distinct from $S$. Then, there exists $T\in[R,S]$ such that $T\subset S$ is minimal ramified, hence not unramified, a contradiction. 
\end{proof}

\begin{proposition}\label{1.5} Let $R\subset S$ be an integral FCP extension.
\begin{enumerate}
\item Assume that  the co-infra-integral closure $S^t_R$ exists. Then, $S^t_R\subseteq T$ and $T\subseteq S$ are infra-integral for any $T\in[S^t_R,S]$. 
\item Assume that  the co-subintegral $S^+_R$ closure exists. Then, $S^+_R\subseteq T$ and $T\subseteq S$ are subintegral for any $T\in[S^+_R,S]$.
\end{enumerate}
\end{proposition} 
\begin{proof} We prove the two parts together. Assume that the co-infra-integral (resp.; co-subintegral) closure exists and set $U:=S^t_R$ (resp.; $U:=S^+_R$). According to Proposition \ref{1.31}, any maximal chain of $ [U,S]$, defined by $U=U_0\subset\cdots\subset U_i\subset\cdots\subset U _n= S$, where each $U_i\subset U_{i+1}$ is minimal is such that each $U_i\subset U_{i+1}$ is either ramified or decomposed (resp.; ramified). Let $T\in [U,S]$. Any maximal chain of $[U,T]$ followed by a maximal chain of $[T,S]$ defines a maximal chain of $ [U,S]$. Then, using again Proposition \ref{1.31}, we get that $U\subseteq T$ and $T\subseteq S$ are infra-integral (resp.; subintegral).  
\end{proof} 

 The following Propositions show the links between the co-subintegral closure, the co-infra-integral closure, the $\omega$-closure, the u-closure 
     and the $\kappa$-radicial closure.

\begin{proposition}\label{1.422} Let $R\subset S$ be  an integral FCP extension. Then 
${}_S^uR={}^\omega_SR\cap{}^t_SR={}^\omega_SR\cap{}^\circ_SR,\ {}^+_RS={}^*_RS\cap{}^\square_SR$ and  ${}^*_RS\cap{}^\omega_SR=R$.

The following statements hold:
\begin{enumerate}
\item ${}_S^uR={}^\omega_SR$ if and only if $R\subseteq{}^\omega_SR$ is infra-integral;

\item If the co-subintegral closure exists, then ${}_S^uR\subseteq S^+_R\cap{}^t_SR$.
\end{enumerate}
\end{proposition} 
\begin{proof} We already know that ${}_S^uR\subseteq{}^t_SR$ by \cite[Definition 1.5 and Proposition 1.6]{Pic 0} and ${}_S^uR\subseteq{}^\omega_SR\subseteq{}^\square_SR$ by Proposition \ref{1.14}; so that, ${}_S^uR\subseteq{}^\omega_SR\cap{}^t_SR$. Since ${}_S^uR\subseteq{}^t_SR$ is both u-closed and infra-integral, any minimal subextension of ${}_S^uR\subseteq{}^t_SR$ is ramified by Proposition \ref{1.311}. Since ${}_S^uR\subseteq{}^\omega_SR$ is unramified, any minimal subextension of ${}_S^uR\subseteq{}^\omega_SR$ is either decomposed or inert separable by Proposition \ref{1.131}. It follows that ${}_S^uR={}^\omega_SR\cap{}^t_SR$. 

Since ${}^t_SR={}^\square_SR\cap{}^\circ_SR$ by Corollary \ref{1.21}, we get that ${}_S^uR={}^\omega_SR\cap{}^\square_SR\cap{}^\circ_SR$. But, ${}^\omega_SR\subseteq{}^\square_SR$ gives ${}_S^uR={}^\omega_SR\cap{}^\circ_SR$. Now, ${}^*_RS\cap{}^\square_SR={}^*_RS\cap{}^\circ_SR\cap{}^\square_SR$ because ${}^*_RS\subseteq{}^\circ_SR$, which leads to ${}^*_RS\cap{}^\square_SR={}^*_RS\cap{}^t_SR$, since we have ${}^t_SR={}^\circ_SR\cap{}^\square_SR$. It follows that ${}^*_RS\cap{}^\square_SR={}^+_RS$ because ${}^*_RS\cap{}^t_SR={}^+_RS$ according to \cite[Corollary 3.6]{Pic 14}. To end, ${}^*_RS\cap{}^\omega_SR={}^*_RS\cap{}^\circ_SR\cap{}^\omega_SR\cap{}^\square_SR={}^*_RS\cap{}^\square_SR\cap{}^\omega_SR\cap{}^t_SR={}_S^uR\cap{}_S^+R$ by the previous equalities. But ${}_S^uR\cap{}_S^+R=R$ according to Proposition \ref{SEP} because $R\subseteq{}_S^uR\cap{}_S^+R$ is both subintegral, and then radicial, and u-integral, then also unramified by Theorem \ref{1.320}. 
    
In particular, ${}_S^uR={}^\omega_SR$ if and only if ${}^\omega_SR\subseteq{}^t_SR$ if and only if $R\subseteq {}^\omega_SR$ is infra-integral.

Moreover, if  the co-subintegral closure exists, $S^+_R\subseteq S$ is subintegral, and then u-closed by Proposition \ref{1.311}. Then, ${}_S^uR\subseteq S^+_R$, which leads to ${}_S^uR\subseteq S^+_R\cap{}^t_SR$. 
 \end{proof}

Let $R\subset S$ be an integral FCP extension (such that the co-infra-integral closure exists for the first diagram). We have the following commutative diagrams:

\centerline{$\begin{matrix}
{} &      {}        & {}_S^+R &      \to      &    {}_S^tR        &       {}       & {}  \\
{} & \nearrow &      {}       & \nearrow &          {}            & \searrow & {}  \\
R &     \to       & {}_S^uR &     \to       & {}^\omega_SR &      \to      & S \\
{} & \searrow &      {}       & \searrow &           {}            & \nearrow & {} \\
{} &     {}        &   S^t_R   &      \to      &      S^+_R         &       {}      & {}
\end{matrix}$ 
and
$\begin{matrix}
{} &      {}       & {}_S^+R &      \to  & {}_S^*R & \to & {}^\circ_SR & {} & {}  \\
{} & \nearrow &      {} & \searrow & {} & \nearrow & {} & \searrow &        {} \\
R &     {}      &       {}       &     {}        &     {}^t_SR        &      {} & {} & {} & S \\
{} & \searrow &     {} & \nearrow & {}  & \searrow & {} & \nearrow &        {} \\
{} &       {} & {}_S^uR & \to & {}^\omega_SR & \to & {}^\square_SR & {} & {}
\end{matrix}$}

\begin{proposition}\label{1.47} Let $R\subset S$ be an integral FCP extension.

\begin{enumerate}
\item Assume that the co-subintegral (resp.; co-infra-integral) closure exists. Then, $S=(S^+_R)({}^+_SR)$ (resp.; $=(S^t_R)({}^t_SR)$). Moreover, $S^+_R\cap{}^+_SR={}^+_{S^+_R}R$ (resp.; $S^t_R\cap{}^t_SR={}^t_{S^t_R}R$).

\item Assume that the co-infra-integral closure exists; so that, the co-subintegral  closure exists. Then, $({}_S^uR)(S^t_R)=S^+_R$.  
\end{enumerate}
\end{proposition} 
\begin{proof} (1) Set $U:=(S^+_R)({}^+_SR)\subseteq S$. Using Proposition \ref{1.5}, we get that $U\subseteq S$ is subintegral and by Proposition \ref{1.31}, we get that $U\subseteq S$ is seminormal; so that, $U=S$. A similar argument shows the equality for the infra-integral case. The two last equalities come from the properties of the seminormalization and the t-closure.

(2) Assume that the co-infra-integral closure exists. According to Corollary  \ref{1.435}, the co-subintegral closure exists and satisfies $S^t_R\subseteq S^+_R$. We know by Proposition \ref{1.422} that ${}_S^uR\subseteq S^+_R$. Then, $({}_S^uR)( S^t_R)\subseteq S^+_R$. Assume that  $({}_S^uR)( S^t_R)\subset S^+_R$ and let $V\in[({}_S^uR)( S^t_R),  S^+_R]$ be such that $V\subset S^+_R$ is minimal. Since $S^t_R\subseteq S$ is infra-integral, $V\subset S^+_R$ is either ramified or decomposed. Since ${}_S^uR\subseteq S$ is u-closed, $V\subset S^+_R$ is either ramified or inert by Corollary \ref{1.315}. Then, $V\subset S^+_R$ is  ramified, and $V\subset S$ is subintegral, a contradiction with the definition of $S^+_R$. To conclude, $({}_S^uR)( S^t_R)=  S^+_R$. \end{proof} 

Thanks to Proposition \ref{1.47}, we are able to give a characterization of extensions whose closures and co-closures play a similar role as in almost-Pr\"ufer extensions.

\begin{corollary}\label{1.501} Let $R\subset S$ be an integral FCP extension. 
\begin{enumerate}
\item Assume that the co-subintegral closure $S^+_R$ of $R\subset S$ exists. Then, the following conditions are equivalent:
\begin{enumerate}
\item$R\subseteq S^+_R$ is seminormal,
\item $R={}_S^+R\cap S^+_R$,
\item $(R:S^+_R)$ is an intersection of finitely many maximal ideals of $S^+_R$.
\end{enumerate}
\item Assume that the co-infra-integral closure $S^t_R$ of $R\subset S$ exists. Then, the following conditions are equivalent:
\begin{enumerate}
\item $R\subseteq S^t_R$ is t-closed,
\item $R={}_S^tR\cap S^t_R$,
\item $(R:S^t_R)$ is an intersection of  finitely many maximal ideals of $S^t_R$ and $|\mathrm{V}_{S^t_R}((R:S^t_R))|=|\mathrm{V}_R((R:S^t_R))|$.
\end{enumerate}
\end{enumerate}  
\end{corollary} 
\begin{proof} (a) $\Leftrightarrow$ (b) in (1) and (2) by Proposition  \ref{1.47}.

\noindent (1) (a) $\Leftrightarrow$ (c)  by Proposition  \ref{1.91}.

\noindent (2) (a) $\Leftrightarrow$ (c) by Proposition \ref{1.91}, Theorem \ref{1.313} and Proposition \ref{1.311}.\end{proof}

\begin{proposition}\label{1.51} Let $R\subset S$ be a locally unbranched integral FCP extension. If the co-infra-integral closure $S^t_R$ exists, then, $S^t_R=S^+_R$. 
\end{proposition} 
\begin{proof} We know that $S^+_R$ exists by Corollary \ref{1.435} and satisfies $S^t_R\subseteq S^+_R$. In order to prove the equality of these rings, it is enough to prove it locally thanks to Proposition \ref{1.426}. Then we may assume that $R$ is local. But the fact that $R\subset S$ is unbranched implies that $R\subset S$ is an $i$-extension, as $S^t_R\subseteq S$, which is then subintegral. Then, $S^t_R=S^+_R$.
\end{proof} 

In \cite[Proposition 4.7]{Pic 12}, we proved that an FCP infra-integral extension is catenarian.

\begin{theorem}\label{1.425} Let $R\subset S$ be an integral FCP extension. Set $T:=\cap[U\in[R,S]\mid U\subseteq S$ is subintegral$]$. Then the following conditions are equivalent:
\begin{enumerate}
\item The co-subintegral closure  of $R\subset S$ exists;

\item $T\subseteq S$ is subintegral;

\item $T\subseteq S$ is  catenarian.
\end{enumerate}
If these conditions hold, then $T$ is the co-subintegral closure  of $R\subset S$.
\end{theorem} 
\begin{proof} Since ${}_S^uR\subseteq U$ for any $U\in[R,S]$ such that $ U\subseteq S$ is subintegral, we get that ${}_S^uR\subseteq T$; so that, $T\subseteq S$ is u-closed, and then an i-extension by Proposition \ref{1.311}.

We prove that (1) $\Rightarrow $ (2), and more precisely, that $T$ is the co-subintegral closure  of $R\subset S$. Indeed, let $T'$ be the co-subintegral closure of $R\subset S$; so that, $T\subseteq T'$. But $T'\subseteq U$ for any $U\in[R,S]$ such that $ U\subseteq S$ is subintegral. Then $T'\subseteq T$ leads to $T'=T$. 

(2) $\Rightarrow$ (3) by \cite[Proposition 4.7]{Pic 12}.

(3) $\Rightarrow$ (1) 
 If $T=S$, there is no $U\in[R,S]$ such that $U\subset S$ is subintegral; so that, $S$ is the co-subintegral closure of $R\subset S$. 

Assume now that $T\neq S$.
Let $U\in[R,S]$ be such that $U\subset S$ is subintegral. Then, $T\subseteq U$; so that, the co-subintegral closure of $R\subset S$, if it exists, is the co-subintegral closure of $T\subseteq S$, and the converse also holds. It follows that, to show the existence of the co-subintegral closure of $R\subset S$, it is enough to show the existence of the co-subintegral closure of $T\subset S$. Using Proposition \ref{1.426}, we are reduced to show that for any $M\in\mathrm{MSupp}_T(S/T)$, there exists a co-subintegral closure of $T_M\subset S_M$. Since $T\subset S$ is catenarian, so is $T_M\subset S_M$ for any $M\in\mathrm{MSupp}(S/T)$ by \cite[Proposition 4.3]{Pic 12}. Let $M\in\mathrm{MSupp}(S/T)$. Then $T_M\subset S_M$ is catenarian and unbranched, and \cite[Theorem 4.13]{Pic 12} asserts that $[T_M,S_M]=[T_M,{}^t_{S_M}T_M]\cup[{}^t_{S_M}T_M,S_M]$. Assume that there exists some $V\in[T_M,S_M]$ such that $V\subset S_M$ is subintegral. We deduce that either $V\subset{}^t_{S_M}T_M\ (*)$ or ${}^t_{S_M}T_M\subseteq V\ (**)$. In case $(*)$, we have ${}^t_{S_M}T_M=S_M$ since $V\subset S_M$ is subintegral and $T_M\subset S_M$ is subintegral; so that, $T_M$ is the co-subintegral closure of $T_M\subset S_M$. But, in case $(**)$, we also have $V\subset S_M$ t-closed; so that, $V=S_M$, a contradiction. If there does not exist some $V\in[T_M,S_M]$  such that $V\subset S_M$ is subintegral, $S_M$ is the co-subintegral closure of $T_M\subset S_M$. In any case, there exists a co-subintegral closure of $T_M\subset S_M$; so that, there exists a co-subintegral closure of $T\subset S$, and also of $R\subset S$. 
\end{proof}

\begin{corollary}\label{1.439} Let $R\subset S$ be an integral FCP extension. If ${}_S^uR\subseteq S$ is catenarian, then the co-subintegral closure  of $R\subset S$ exists. 
In particular, these conditions hold if $R\subset S$ is  catenarian.
 
\end{corollary} 
\begin{proof} We use the ring $T:=\cap[U\in[R,S]\mid U\subseteq S$ is subintegral$]$ defined in Theorem \ref{1.425}. It was proved in this Theorem that ${}_S^uR\subseteq T$; so that, $T\subseteq S$ is catenarian if so is ${}_S^uR\subseteq S$. Then, Theorem \ref{1.425} shows the existence of the co-subintegral closure  of $R\subset S$.
\end{proof}

This Corollary shows in another way than in Theorem \ref{1.423} that an FCP infra-integral extension admits a co-subintegral closure.

Although we have  similar statements for the subintegral case (Theorem \ref{1.425}) and the infra-integral case (Theorem \ref{1.628}), this last case needs some lemmata  before getting the result. To this aim, we begin by considering an FCP extension $R\subset S$ with $T,U\in[R,S],\ T\neq U$ such that $T\subset S$ and $U\subset S$ are both minimal non-inert and look at the behavior of $T\cap U\subset S$.

\begin{lemma}\label{1.430} Let $R\subset S$ be an integral FCP extension. Assume that there exist $T,U\in[R,S],\ T\neq U$ such that $T\subset S$ and $U\subset S$ are both minimal non-inert with the same crucial maximal ideal $M$. Set $V:=T\cap U$. The following properties hold:
\begin{enumerate}
\item $T\subset S$ and $U\subset S$ are  minimal of the same type.

\item $M\in\mathrm{Max}(V)$ and $M=(V:S)$.

\item $V\subset T$ and $V\subset U$ are t-closed finite extensions.

\item If $T\subset S$ and $U\subset S$ are both ramified, then $R\subset S$  admits  neither a co-subintegral closure nor a co-infra-integral closure.
 \end{enumerate} 
\end{lemma} 
\begin{proof} (1) According to \cite[Proposition 5.7]{DPPS}, $T\subset S$ and $U\subset S$ are minimal extensions of the same type and $M\in\mathrm{Max}(V)$. 

(2)  Since $M$ is also an ideal of $S$, it follows that $M=(V:S)$.

(3) We get that $M\in\mathrm{Max}(V),\ \mathrm{Max}(T)$ and $M\in\mathrm{Max}(U)$; so that, $V/M\subset T/M$ and $V/M\subset U/M$ are finite fields extensions, and then t-closed extensions by \cite[Lemme 3.10]{Pic 1}, and so are $V\subset T$ and $V\subset U$ by \cite[Th\'eor\`eme 3.15]{Pic 1}. In particular, $V\subset S$ is not infra-integral.

(4) Assume that $R\subset S$ admits a co-subintegral closure $S^+_R$. If $T\subset S$ and $U\subset S$ are both minimal ramified, $T,U\in[S^+_R,S]$ implies that $V=T\cap U\in[S^+_R,S]$, a contradiction by (3) since $S^+_R\subset S$ is subintegral. It follows that the co-infra-integral closure also does not exists by  Corollary \ref{1.435}.
\end{proof}

\begin{lemma}\label{1.467} Let $R\subset S$ be an integral  catenarian FCP extension. Let  $T,U\in[R,S],\ T\neq U$ be such that $T\subset S$ and $U\subset S$ are minimal non-inert. Then $T\cap U\subset S$ is infra-integral and $\ell[T\cap U,S]=2$. 
\end{lemma} 
\begin{proof} Set $V:=T\cap U,\ M:=(T:S)$ and $N:=(U:S)$. Then, $M$ and $N$ are ideals of $S$. Assume first that $M\neq N$. According to \cite [Proposition 6.6]{DPPS}, $V\subset T$ and $V\subset U$ are both minimal non-inert; so that, $V\subset S$ is infra-integral. 

Assume now that $M=N$. By Lemma \ref{1.430}, we get, that $V\subset T$ is t-closed. Since $R\subset S$ has FCP, there exists $T'\in[V,T]$ such that $T'\subset T$ is minimal inert by Proposition \ref{1.31}. Moreover, $M=(T':T)=(T:S)$, a contradiction given by \cite[Lemma 4.10]{Pic 12}, since $R\subset S$ is catenarian.

To conclude, the only possible case is that $M\neq N$ and $T\cap U\subset S$ is infra-integral with $\ell[V,S]=2$.
\end{proof}

\begin{lemma}\label{1.469} Let $R\subset S$ be a catenarian FCP  integral extension. 
Let $T,U\in[R,S]$ be such that $T\subset S$ is infra-integral and $U\subset S$ is minimal non-inert. Then $T\cap U\subset S$ is infra-integral. 
  If, in addition, $T\not\subseteq U$, then $T\cap U\subset T$ is minimal non-inert.
 \end{lemma} 
\begin{proof} 
 Assume first that $T\subseteq U$. Then, $T\cap U=T$; so that, $T\cap U\subset S$ is infra-integral.

Assume now that $T\not\subseteq U$ and set $V:=T\cap U$. 

We will prove that $V\subset S$ is infra-integral and $V\subset T$ is minimal non-inert by induction on $r:=\ell[T,S]<\ell[R,S]$. For $r=1$, we have $\ell[T,S]=1$; so that, $T\subset S$ and $U\subset S$ are minimal non-inert. Then, Lemma \ref{1.467} gives the result because $T\neq U$. 

Assume that the result holds for some $1\leq r\leq\ell[R,S]-2$. Let $T,U\in[R,S],\ T\not\subseteq U$ be such that $T\subset S$ is infra-integral and $U\subset S$ is minimal non-inert, with $\ell[T,S]=r+1$. Since $r\geq 1$, $T\subset S$ is not minimal; so that, there exists $T'\in[T,S]$ such that $T\subset T'$ is minimal non-inert. Since $R\subset S$ is catenarian, we have $\ell[T,S]=\ell[T,T']+\ell[T',S]$, which implies that $\ell[T',S]=r$, and the induction hypothesis holds for $T'\subset S$ and $U\subset S$ because $T'\not\subseteq U$. Setting $V':=T'\cap U$, we get that $V'\subset S$ is infra-integral, with $V'\subset T'$ minimal (non-inert). We have the following commutative diagram:
\centerline{$\begin{matrix}
     T        & \to &      T'       & \to &      S        \\
\uparrow & {}  & \uparrow & {}   & \uparrow  \\
    V        & \to &      V'       & \to &      U
\end{matrix}$} 
Applying Lemma \ref{1.467} to the extensions $T\subset T'$ and $V'\subset T'$, it follows that $T\cap V'=T\cap T'\cap U=T'\cap V=V\subset T'$ is infra-integral, and so is $V\subset S$. In fact, $V'\neq T$. Otherwise, $V=T$ leads to $T\subseteq U$, a contradiction. Moreover, $V\subset T$ is minimal (non-inert). 
 Then, the induction hypothesis holds for $r+1$, and then for any $r$. 
\end{proof}

 \begin{lemma}\label{1.468} Let $R\subset S$ be an integral  catenarian FCP extension. If $T,U\in[R,S]$ are such that $T\subset S$ and $U\subset S$ are infra-integral, then $T\cap U\subset S$ is infra-integral. 
\end{lemma} 
\begin{proof} Set $n:=\ell[U,S]$. There exists a maximal chain $\{U_i\}_{i=0}^n$ such that $U_0:=S$ and $U_n:=U$, with $U_{i+1}\subset U_i$ minimal infra-integral. For each $i\in\mathbb N_n$, set $V_i:=T\cap U_i$. We  prove by induction on $i\in\mathbb N_n$ that $V_i\subset S$ is infra-integral for each $i\in\mathbb N_n$. 

For $i=1$, Lemma \ref{1.469} gives the result. 

Assume that the induction hypothesis holds for some $i\in\mathbb N_{n-1}$; so that, $V_i\subset S$ is infra-integral. In particular, $V_i\subseteq U_i$ is infra-integral and $U_{i+1}\subset U_i$ is minimal infra-integral. As $V_{i+1}=T\cap U_{i+1}$ and $U_{i+1}\subset U_i$, we get that $V_{i+1}=T\cap U_{i+1}\cap U_i=V_i\cap U_{i+1}$. Applying Lemma \ref{1.469} to the extensions $V_i\subseteq U_i$ and $U_{i+1}\subset U_i$, we get that $V_{i+1}\subseteq U_i$ is infra-integral, and so is $V_{i+1}\subset S$. To conclude, the induction hypothesis holds for any $i\in\mathbb N_n$. In particular, for $i=n$, it follows that $V_n=T\cap U\subset S$ is infra-integral.
\end{proof}

\begin{theorem}\label{1.628} Let $R\subset S$ be an integral FCP extension. Set $T:=\cap[U\in[R,S]\mid U\subseteq S$ is infra-integral $]$. Then the following conditions are equivalent:
 \begin{enumerate}
\item The co-infra-integral closure  of $R\subset S$ exists;

\item $T\subseteq S$ is infra-integral;

\item $T\subseteq S$ is catenarian.
\end{enumerate}
If these conditions hold, then $T$ is the co-infra-integral closure of $R\subset S$.
\end{theorem} 
\begin{proof} 
(1) $\Rightarrow$ (2). We prove first that $T$ is the co-infra-integral closure of $R\subset S$. Indeed, let $T'$ be the co-infra-integral closure of $R\subset S$; so that, $T\subseteq T'$. But $T'\subseteq U$ for any $U\in[R,S]$ such that $U\subseteq S$ is infra-integral. Hence, $T'\subseteq T$ leads to $T'=T$ and then (1) $\Rightarrow$ (2). 

(2) $\Rightarrow$ (1) Since any $U\in[R,S]$ such that $U\subseteq S$ is infra-integral verifies $T\subseteq U$, it follows that $T$ is the co-infra-integral closure of $R\subset S$.

(2) $\Rightarrow$ (3) by \cite[Proposition 4.7]{Pic 12}.

(3) $\Rightarrow$ (2) Since $R\subset S$ has FCP, set $m:=\ell[T,S]$. If $T=S$, there is no $U\in[R,S]$ such that $U\subset S$ is infra-integral; so that, $S$ is the co-infra-integral closure of $R\subset S$. Assume that $m\geq 1$. Then, $T\subset S$. We show by induction on $k\in\mathbb N_m$ that there exists $U_k\in[T,S]$ such that $U_k\subset S$ is infra-integral with $\ell[U_k,S]=k$. Since $m\geq 1$, there is some $U_1\in[T,S[$ such that $U_1\subset S$ is minimal infra-integral; so that, the induction hypothesis holds for $k=1$. Assume that the induction hypothesis holds for some $k\in\mathbb N_{m-1}$; so that, there exists $U_k\in[T,S]$ such that $U_k\subset S$ is infra-integral with $\ell[U_k,S]=k$. In particular, $T\subset U_k$ since $k<m$. Assume that for any $V\in[R, S]$ such that $V\subset S$ is infra-integral we have $U_k\subseteq V$. Then, $T=U_k$, a contradiction. This implies that there exists some $V\in[R, S]$ such that $V\subset S$ is infra-integral, with $U_k\not\subseteq V$, giving $T\subseteq V\cap U_k\subset U_k$. By Lemma \ref{1.468}, $V\cap U_k\subset U_k$ is infra-integral, and there exists $U_{k+1}\in[V\cap U_k,U_k]$ such that $U_{k+1}\subset U_k$ is minimal infra-integral; so that, $\ell[U_{k+1},S]=k+1$. Then, the induction hypothesis holds for any $k\in\mathbb N_m$. Of course, for $k=m$, we get that $U_m=T$, and $T\subset S$ is infra-integral. 
\end{proof}

\begin{corollary}\label{1.629} Let $R\subset S$ be an integral catenarian FCP extension. Then the co-infra-integral closure  of $R\subset S$ exists.
\end{corollary} 
\begin{proof} We use the ring $T:=\cap[U\in[R,S]\mid U\subseteq S$ is infra-integral$]$ defined in Theorem \ref{1.628}. Then $T\subseteq S$ is catenarian and Theorem \ref{1.628} shows the existence of the co-infra-integral closure  of $R\subset S$.
\end{proof}

\begin{proposition}\label{1.427} Let $R\subset S$ be an integral FCP extension with co-subintegral (resp.; co-infra-integral) closure $T$. For any $U\in[R,S],\ TU$ is  the co-subintegral (resp.; co-infra-integral) closure  of $U\subseteq S$, that is $(S^+_R)U=S^+_U$ (resp.; $(S^t_R)U=S^t_U$).
\end{proposition} 
\begin{proof} Since $TU\in[T,S]$, we get that $TU\subseteq S$ is subintegral (resp.; infra-integral), with $TU\in[U,S]$. Let $V\in[U,S]$ be such that $V\subseteq S$ is subintegral (resp.; infra-integral). It follows that $T\subseteq V$; so that, $TU\subseteq V$, and $TU$ is the least $V\in[U,S]$ such that $V\subseteq S$ is subintegral (resp.; infra-integral). Then, $TU$ is the co-subintegral (resp.; co-infra-integral) closure of $U\subseteq S$. 
\end{proof}

\begin{remark}\label{1.428}  Let $R\subset S$ be an integral FCP extension with co-subintegral closure $T$ and let $U\in[R,S]$. Contrary to other closures like integral closure, seminormalization or t-closure, there is no reason for $U\cap T$ to be the co-subintegral closure of $R\subseteq U$ because $U\cap T$ is not necessarily the least $V\in[R,T]$ such that $U\cap T\subseteq U$ is subintegral. The example given in Remark \ref{1.317} and Theorem \ref{1.423} show that $R^2$ is the co-subintegral closure of $R\subset S$. But $R^2\cap{}^+_SR={}^+_{R^2}R$ is not the co-subintegral closure of $R\subset{}^+_SR$ since $R\subset{}^+_SR$ is subintegral. In this example, since $R\subset S$ is infra-integral, $R$ is both the co-infra-integral closure of $R\subset S$ and $R\subset{}^+_SR$. 
\end{remark} 

 We recall a result from \cite{Pic 15}.

\begin{proposition}\label{1.429} \cite[Proposition 5.21]{Pic 15} An integral FCP extension $R\subset S$, that is split at ${}^+_SR$ (resp.; ${}^t_SR$), has a co-subintegral (resp.; co-infra-integral) closure $S^+_R$ (resp.; $S^t_R$), which is its complement.
   \end{proposition} 

\begin{remark}\label{1.436} 
(1) The last result of Proposition \ref{1.429} could also be gotten by Proposition \ref{1.47}.

(2) Let $R\subset S$ be an integral FCP extension such that the co-subintegral closure $S^+_R$ exists and is a complement of ${}^+_SR$. Then $R\subset S$ does not need to  split at ${}^+_SR$. See the example of Remarks \ref{1.428} and \ref{1.317}, and take $R_1:=R+(M\times M)$. Then, ${}^+_SR_1={}^+_SR$ and $S^+_{R_1}=S^+_R=R^2$. We get ${}^+_SR_1\cap S^+_{R_1}=R_1$ and $({}^+_SR_1)(S^+_{R_1})=S$. But $R_1$ is a local ring; so that, $R_1\subset S$ does not split at ${}^+_SR_1$.
\end{remark}

We end this section by giving some examples where we can check if the co-subintegral closure exists or not.  

The following Lemmata are special cases of Lemma \ref{1.430} and explain more precisely the situation of Remark \ref{1.421}.

\begin{lemma}\label{1.431} Let $K\subset S$ and $L\subset S$ be two minimal ramified extensions where $K$ and $L$ are distinct fields. Set $k:=K\cap L$. The following properties hold:
\begin{enumerate}
\item There exists $x\in S,\ x\neq 0$ such that $x^2=0$ with $S=K+Kx=L+Lx$, which is a local ring. Moreover, $M:=Kx=Lx$ is the maximal ideal of $S$. 

\item There exists a $k$-isomorphism of $k$-algebras $\varphi:K\to L$ defined, for any $a\in K$, by $\varphi(a)=b$, where $b$ is the unique element of $L$ such that $ax=bx$.

\item $k$ is a field and the extension $k\subset S$ does not admit a co-subintegral closure.
\end{enumerate} 
\end{lemma} 
\begin{proof} (1) Since $K\subset S$ is ramified, there exists $x\in S,\ x\neq 0$ such that $x^2=0$ with $S=K+Kx$, which is a local ring with maximal ideal $M:=Kx$ by Theorem \ref{minimal}. Since $x^2=0$, we have $x\not\in L$ because $L$ is a field; so that, $S=L[x]$. Then, $S=L+Lx$ and  $M=Lx$. In particular, $Kx=Lx\ (*)$. 

(2) Let $a\in K$. There exists $b\in L$ such that $ax=bx$ by $(*)$. We claim that this $b$ is unique for such $a$. Let $b'\in L$ such that $ax=b'x$; so that, $ax=bx=b'x$, which leads to $(b-b')x=0$.  In particular, $b-b'\in M$. If not, $b-b'$ is a unit of $S$ since $S$ is a local ring, which leads to $x=0$, a contradiction with (1). Then, $b-b'\in M\cap L=\{0\}$; so that, $b=b'$. Define $\varphi:K\to L$ by $\varphi(a)=b$, for any $a\in K$, where $b$ is the unique element of $L$ such that $ax=bx$. If $a\in k=K\cap L$, the equality $ax=ax$ shows that $\varphi(a)=a$, giving that $\varphi$ is the identity on $k$. For $a,a'\in K$, we obviously have $\varphi(a+a')=\varphi(a)+\varphi(a')$. Set $b:=\varphi(a)$ and $b':=\varphi(a')$; so that, $ax=bx$ and $a'x=b'x$. Then, $aa'x=ab'x=b'ax=b'bx$ leads to $\varphi(aa')=\varphi(a)\varphi(a')$. In particular, this relation shows that $\varphi$ is a $k$-algebra injective morphism, because between two fields, and is obviously surjective. To conclude, $\varphi$  is a $k$-algebra isomorphism.

(3) $k$ is obviously a field, and $k=K\cap L\subset S$ is not subintegral since there is a residual extension $k\subset S/M\cong K$, which is not an isomorphism. Anyway, when $k\subset S$ has FCP, Lemma \ref{1.430}(4) shows that $k\subset S$ does not admit a co-subintegral closure since $(K:S)=(L:S)=0$, the maximal ideal of both $K$ and $L$.
\end{proof}

\begin{lemma}\label{1.433} Let $K\subset S$ and $L\subset S$ be two minimal decomposed extensions where $K$ and $L$ are distinct fields. Set $k:=K\cap L$. The following properties hold:
\begin{enumerate}
\item There exists $x\in S,\ x\neq 0$ such that $x^2=x$ with $S=K+Kx=L+Lx$. Moreover, $M_1:=Kx=Lx$ and $M_2:=K(1-x)=L(1-x)$ are the  maximal ideals of $S$. 

\item There exists a $k$-isomorphism of $k$-algebras $\varphi:K\to L$ defined, for any $a\in K$, by $\varphi(a)=b$, where $b$ is the unique element of $L$ such that $ax=bx$.

\item $k$ is a field and the extension $k\subset S$ does not admit a co-infra-integral closure.
\end{enumerate} 
\end{lemma} 
\begin{proof} (1) Since $K\subset S$ is decomposed, there exists $x\in S,\ x\neq 0,1$ such that $x^2=x$ with $S=K+Kx$,  which has maximal ideals $M_1:=Kx$ and $M_2:=K(1-x)$ by Theorem \ref{minimal}. Since $x^2-x=0$, we have $x\not\in L$ since $L$ is a field; so that, $S=L[x]$. Then, $S=L+Lx$ and  $M_1=Lx$ because $x\in M_1$. In a similar way, $M_2:=L(1-x)$. In particular, $Kx=Lx\ (*)$. 

(2) Let $a\in K$. There exists $b\in L$ such that $ax=bx$ by $(*)$. As in Lemma \ref{1.431}, we claim that this $b$ is unique for such $a$. Let $b'\in L$ such that $ax=b'x$; so that, $ax=bx=b'x$, which would lead to $(b-b')x=0$. In particular, $b-b'$ is not a unit of $S$, which leads to $x=0$, a contradiction with (1). Then, $b-b'\in(M_1\cup M_2)\cap L=\{0\}$; so that, $b=b'$. Define $\varphi:K\to L$ by $\varphi(a)=b$, for any $a\in K$, where $b$ is the unique element of $L$ such that $ax=bx$. If $a\in k=K\cap L$, the equality $ax=ax$ shows that $\varphi(a)=a$, giving that $\varphi$ is the identity on $k$. For $a,a'\in K$, we obviously have $\varphi(a+a')=\varphi(a)+\varphi(a')$. Set $b:=\varphi(a)$ and $b':=\varphi(a')$; so that, $ax=bx$ and $a'x=b'x$. Then, $aa'x=ab'x=b'ax=b'bx$ leads to $\varphi(aa')=\varphi(a)\varphi(a')$. In particular, this relation shows that $\varphi$ is a $k$-algebra injective morphism, because between two fields, and is obviously surjective. To conclude, $\varphi$ is a $k$-algebra isomorphism.

(3) $k$ is obviously a field, and $k=K\cap L\subset S$ is not infra-integral since there are two residual extensions $k\subset S/M_i\cong K,\ i=1,2$, which are not an isomorphism. Anyway, if $k\subset S$ has FCP, since $(K:S)=(L:S)=0$, the maximal ideal of both $K$ and $L$, a proof similar as the one of Lemma \ref{1.430} (4) shows that $k\subset S$ does not admit a co-infra-integral closure. 
\end{proof}

We remark that there do not exist $K\subset S$ and $L\subset S$ which are two minimal infra-integral extensions of different types, where $K$ and $L$ are distinct fields. Indeed, if $K\subset S$ is minimal ramified, then $S$ is local, while it has two maximal ideals if $L\subset S$ is minimal decomposed, a contradiction. See also \cite[Proposition 5.7]{DPPS}.

 The situation of Lemma \ref{1.433} is satisfied in the following example.
 
\begin{example}\label{1.434} Let $K:=\mathbb Q[j]$, where $j=(-1+i\sqrt 3)/2$, and $S:=K[X]/(X^2-X)=K[x]$, where $x$ is the class of $X$ in $S$ and satisfies $x^2=x$. Then, $K\subset S$ is a minimal decomposed extension, hence infra-integral, and $\mathbb Q\subset K$ is a minimal separable field extension of degree 2, and then t-closed. Set $\alpha:=-j+(1+2j)x\in S\setminus K$ because $x\not\in K$. A short calculation shows that $\alpha$ is a zero of the irreducible polynomial $X^2-X+1\in \mathbb Q[X]$; so that, $L:=\mathbb Q[\alpha]$ is a field, with $L\in[\mathbb Q,S],\ L\neq K$, because $\alpha\not\in K$, and $\mathbb Q\subset L$ a minimal separable field extension of degree 2. Moreover, $x\not\in L$ since $L$ is a field and $x^2-x=0$. 
  Since $(2x-1)^2=1$, this shows that $2x-1$ is a unit of $L[x]$. Then, $j=(\alpha-x)(2x-1)^{-1}\in L[x]$; so that, $S=L[x]$. 
It follows that $L\subset S$ is a minimal decomposed extension, and then infra-integral. As $\mathbb Q=K\cap L$, Lemma \ref{1.433} shows that $\mathbb Q\subset S$ does not admit a co-infra-integral closure.  
\end{example} 

By Lemma \ref{1.430}, we get that an FCP extension of the form $T\cap U\subset S$, where $T\subset S$ and $U\subset S$ are minimal ramified with the same crucial maximal ideal does not admit a co-subintegral closure. If $(T:S)\neq (U:S)$, then \cite[Proposition 6.6]{DPPS} asserts that $T\cap U\subset S$ is subintegral, and then $T\cap U$ is the co-subintegral closure of $T\cap U\subset S$. The following Proposition exhibits  cases where situation of Lemma \ref{1.430} does not occur.  

\begin{proposition}\label{1.432} Let $R\subset S$ be an unbranched integral FCP extension, where $(R,P)$ and $(S,N)$ are local rings. Assume that $R/P\subset S/N$ is a finite separable field extension. Then, there does not exist $T,U\in[R,S],\ T\neq U$ such that $T\subset S$ and $U\subset S$ are minimal ramified with the same crucial maximal ideal.
\end{proposition} 

\begin{proof} Assume that there exist $T,U\in[R,S],\ T\neq U$ such that $T\subset S$ and $U\subset S$ are minimal ramified with the same crucial maximal ideal $M$. Set $V:=T\cap U$. According to Lemma \ref{1.430}, $V\subset T$ and $V\subset U$ are t-closed finite extensions such that $M=(V:S)=(V:T)=(V:U)\in\mathrm{Max}(V)$. Set $k:=V/M,\ K:=T/M,\ L:=U/M\neq K,\ S':=S/M$ and $N':=N/M$. Since $S/N\cong K\cong L$, we get that $R/P\subset K$ and $R/P\subset L$ are finite separable field extensions, and so are $k\subset K$ and $k\subset L$. Then, there exists $\alpha\in K$ such that $K=k[\alpha]$. Moreover, $K\subset S'$ and $L\subset S'$ are distinct minimal ramified extensions. We are in the situation of Lemma \ref{1.431}; so that, there exists $x\in S',\ x\neq 0$ such that $x^2=0$ with $S'=K[x]=L[x]$, where $N'=Kx=Lx$ is the maximal ideal of $S'$. In particular, $N'^2=0$. Furthermore, there exists a $k$-isomorphism of $k$-algebras $\varphi:K\to L$ defined, for any $a\in K$, by $\varphi(a)=b$, where $b$ is the unique element of $L$ such that $ax=bx$. Set $\beta:=\varphi(\alpha)$, which gives $L=k[\beta]$. Since $\alpha x=\beta x$ in $S'$, we have $(\alpha-\beta)x=0$; so that, $y:=\alpha-\beta\in N'$ satisfies $y^2=0$ with $y\neq 0$. Let $Q(X):=\sum_{i=0}^na_iX^{i}\in k[X]$ be the minimal monic polynomial of $\alpha$, which is also the minimal monic polynomial of $\beta$ through $\varphi$. Then, $Q(\alpha)=Q(y+\beta)=0=\sum_{i=0}^na_i(y+\beta)^{i}=
\sum_{i=0}^na_i\beta^{i}+\sum_{i=0}^nia_iy\beta^{i-1}=Q(\beta)+yQ'(\beta)=yQ'(\beta)$, with $Q'(\beta)\in k[\beta]=L$, which is a field. It follows that $Q'(\beta)=0$ since $y\neq 0$. But $Q(X)$ is a separable polynomial, a contradiction with $Q(X)$ the minimal polynomial of $\beta$. To conclude, our assumptions do not hold and there do not exist $T,U\in[R,S],\ T\neq U$ such that $T\subset S$ and $U\subset S$ are minimal ramified with the same crucial maximal ideal.
\end{proof}

Contrary to Proposition \ref{1.432}, Example \ref{1.434} shows that there exist $T,U\in[R,S],\ T\neq U$ such that $T\subset S$ and $U\subset S$ are minimal decomposed with the same crucial maximal ideal, where $(R,P)$ is a local ring, $\mathrm{Max}(S)=\{M_1,M_2\}$ and $R/P\subset S/M_i$ is a finite separable field extension for $i=1,2$. It is enough to take $R:=\mathbb Q,\ T:=K$ and $U:=L$; so that, $S/M_i\cong K\cong L$ for $i=1,2,\ R/P=\mathbb Q$ and $R/P\subset S/M_i$ is a finite separable field extension for $i=1,2$.  

\section{Cardinality of the lattices of  integral FIP extensions}

Let $R\subset S$ be an integral FIP extension. We are going to give information about $|[R, S]|$. Because $R\subset S$ is a $\mathcal B$-extension, it is enough to assume that $R$ is a local ring.  In order to examine the elements of $[R,S]$, we use the different closures (the t-closure and the seminormalization) we studied in the previous sections to get bounds for $|[R,S]|$. 

Let $R\subset S$ be an integral FCP  extension and  $T$ its   seminormalization (resp.; t-closure). When it exists let $T'$ be its co-subintegral (resp.; co-infra-integral) closure. Define the map $\varphi:[R,S]\to[R,T]\times [T,S]$ by $\varphi(U):=(T\cap U, TU)$ for any $U\in[R,S]$. Recall that $T^o$ denotes the complement of $T$ in $[R,S]$ when it exists and is unique. The following Proposition  gives relations between $\varphi$, the existences of a complement of $T$ and $T'$. 

\begin{proposition}\label{A4.0} Let $R\subset S$ be an integral FCP. Consider the following conditions:
\begin{enumerate}
\item $\varphi$ is bijective.
\item $\varphi$ is surjective.
\item $T$ admits a complement.
\item $T^o$ exists.
\item $T'$ exists.
\end{enumerate}
 We have the following implications: (1) $\Rightarrow$ (2) $\Rightarrow$ (3), (1) $\Rightarrow$ (4) $\Rightarrow$ (3) and (3)+(5)  $\Rightarrow$ (4)
with $T'=T^o$ in this last case.  
\end{proposition}

\begin{proof} We make the proof for $T={}_{S}^+R$. 

 (1) $\Rightarrow$ (2) is obvious.

(2) $\Rightarrow$ (3): Since $\varphi$ is surjective, there exists $U\in[R,S]$ such that $\varphi(U)=(R,S)=(T\cap U,TU)$; so that, $T\cap U=R$ and $ TU=S$, which shows that $U$ is a complement of $T$.

(1) $\Rightarrow$ (4): Since $\varphi$ is bijective, the previous proof shows that $T$ has a unique complement.

(3) $+$ (5) $\Rightarrow$ (4): Assume that $T$ has a  complement $V$ and that the co-subintegral closure $T'$ exists. By definition, $TV=S$ implies that $V\subseteq S$ is subintegral; so that, $T'\subseteq V$. Assume that $T'\neq V$. Since $R\subseteq V$ is seminormal, so is $T'\subseteq V$ and $T'\subseteq V$ is subintegral because so is $T'\subseteq S$. Then, $T'=V$, and $T$ has a unique complement  $T^o=T'$.

When $T={}_{S}^tR$, it is enough to replace in the previous proof seminormal with t-closed and subintegral with infra-integral.
\end{proof}

We begin to consider integral u-closed extensions.

\begin{lemma}\label{A5.0} Let $R\subset S$ be an integral FCP u-closed extension where $(R,M)$ is a local ring such that $R\neq {}_S^+R\neq S$. Assume that $R\subseteq S'$ is t-closed for any $S'\in[R,S]$  such that $MS'$ is the maximal ideal of $S'$. Then, ${}_{S}^+R$ has a complement. If, in addition, this complement is unique, then it is the co-subintegral closure of $R\subset S$, that is $({}_{S}^+R)^o=S^+_R$.
 \end{lemma}

\begin{proof} Recall that ${}_S^+R={}_S^tR$ and $\mathrm{Spec}(S)\to\mathrm{Spec}(R)$ is a bijection by Proposition \ref{1.311}. Moreover, $R\subset S$ is unbranched and only ramified and inert minimal extensions appear in a maximal chain of extensions composing $R\subset S$. Set $\mathcal F:=\{U\in[R,S]\mid U\subseteq S$ is subintegral$\}$. We have $\mathcal F\neq\emptyset$ because $S\in\mathcal F$. Since $R\subset S$ has FCP, there exists some $U_0\in\mathcal F$ such that $\ell[U_0,S]$ is maximal; so that, there does not exists any $V\in[R,U_0]$ such that $V\subset U_0$ is minimal ramified. As $R\subset S$ is not subintegral, $R\neq U_0$. Then, according to \cite[Lemma 17]{DPP4}, $MU_0$ is a radical ideal of $U_0$, that is an intersection of maximal ideals of $U_0$. But, $U_0$ is a local ring; so that, $MU_0$ is the maximal ideal of $U_0$, giving that $R\subseteq U_0$ is t-closed by assumption. Then, $U_0\neq S$ because $R\subset S$ is not t-closed. As $U_0\subset S$ is subintegral, we get that ${}_{S}^+R\cap U_0=R$ and ${}_{S}^+RU_0=S$; so that, $U_0$ is a complement of ${}_{S}^+R$.

Assume, in addition, that this complement is unique and let $T$ be this complement. Then, $T\in \mathcal F$. Of course, ${}_{S}^+R\cap T=R$ and ${}_{S}^+RT=S$ show that $R\subseteq T$ is t-closed and $T\subseteq S$ is subintegral. Set $T':=\cap [U\in \mathcal F]$; so that, $T' \subseteq T$. If $T'\neq T$, there exists some $V\in\mathcal F$ with $V\not\in[T,S]$. Choose such $V\not\in [T,S]$ satisfying that $\ell[V,S]$ is maximal. As we proved in the beginning of the  proof for $U_0$, we get that $V$ is a complement of ${}_{S}^+R$, a contradiction with the uniqueness of this complement. Then, $T'=T$ implies that $T'\subseteq S$ is subintegral; so that, $T'$ is the co-subintegral closure of $R\subset S$ by Theorem \ref{1.425}. 
 \end{proof}

\begin{remark}\label{A5.1} Lemma \ref{1.430} shows that when ${}_{S}^tR$ has more than one complement, then the co-subintegral closure of $R\subset S$ does not exists.
\end{remark}

Contrary to Proposition \ref{1.432}, where separable field extensions appear, here is an example illustrating Lemma \ref{A5.0} 
 in the context of radicial field extensions.

\begin{example}\label{A5.2} Let $k$ be a field with $\mathrm{c}(k)=2$ and let $K:=k[t]$ be a radicial field extension of degree 2 such that $t^2=a\in k$. Set $S:=K[X]/(X^2)=K[x]$, where $x$ is the class of $X$ in $S$, so that $K\subset S$ is a minimal ramified extension and set $\alpha:=t+x$. Then, $\alpha^2=t^2=a\in k$; so that, $L:=k[\alpha]$ is a radicial field extension of $k$ of degree 2, with $Y^2-a$ the minimal polynomial of $\alpha$ and $L$ is a field different from $K$ since $x\not\in K$.
 Moreover, $L\subset S$ is a minimal ramified extension and $k\subset S$ has FCP. This example illustrates also Lemma \ref{1.431}. The maximal ideal of $S$ is $Kx=Lx$. Since $k$ is a field, its maximal ideal is $0$, and for any $S'\in[R,S]$ such that $0S'=0$ is the maximal ideal of $S'$, we have that $S'$ is a field; so that, $k\subseteq S'$ is t-closed. Then, $k\subseteq S$ satisfies the assumptions of Lemma \ref{A5.0}. In particular, $k':={}_{S}^tk=k+Kx=k+Lx$ because $k'/Lx\cong k$ shows that $k\subset k'$ is subintegral, while $k'\subset S$ is t-closed because $(k':S)=Lx$, with $k\cong k'/Lx\subset S/Lx\cong L$ a field extension such that $L\cong K$. To end, we remark that $K$ and $L$ are both complements of ${}_{S}^tk$. 
\end{example}

In case $R\subset S$ is an integral u-closed FIP extension, the following Proposition shows that when ${}_S^+R$ has a complement, it is unique and equal to the co-subintegral closure of $R\subset S$ when this last one  exists.

\begin{proposition}\label{B5.1} Let $R\subset S$ be an integral u-closed FIP extension. Assume that $(R,M)$ is a local ring and $R\neq {}_S^+R\neq S$.  

\begin{enumerate}
\item If ${}_S^+R$ has a complement, then $MS\not\in\mathrm{Max}(S)$. Moreover, this complement is unique. 
If, in  addition, $S^+_R$ exists, then $({}_{S}^+R)^o=S^+_R$.

\item The map $\varphi:[R,S]\to[R,{}_S^+R]\times[{}_S^+R,S]$ defined by $\varphi(U)=(U\cap{}_S^+R,{}_S^+RU)$ for $U\in[R,S]$ is injective. In particular, $|[R,{}_S^+R]|+|[{}_S^+R,S]|-1\leq|[R,S]|\leq|[R,{}_S^+R]||[{}_S^+R,S]|$.
\end{enumerate}
\end{proposition}

\begin{proof} (1) Let $T'\in[R,S]$ be a complement of ${}_S^+R$. Since $R=T'\cap{}_S^+R$ and $S={}_S^+RT'$, we have $T'\neq R,S$ because $R\subset S$ is neither subintegral nor t-closed. In particular, $T'\subset S$ is subintegral; so that, there exists some $U\in[T',S]$ such that $U\subset S$ is minimal ramified. Then, \cite[Lemma 17]{DPP4} shows that  $MS$ is not a radical ideal of $S$, and in particular, $MS\not\in\mathrm{Max}(S)$.  

 Let $N$ be the maximal ideal of $S$, which is a local ring because $R\subset S$ is u-closed. 
Assume now that there exist another complement $T''\in[R,S]$ of ${}_S^+R$. Set $T:=T'T''\subseteq S$. Hence, $M$ is also the maximal ideal of $T'$ and $T''$ since $R\subset T'$ and $R\subset T''$ are t-closed, and also an ideal of $T$, a local ring since $R\subset S$ is an i-extension. Set $k:=R/M,\ K':=T'/M,\ K'':=T''/M$ and $K:=T/M=K'K''$; so that, $k,K'$ and $K''$ are fields, with $K',K''\cong S/N$, because $T'\subset S$ and $T''\subset S$ are subintegral. Moreover, $k\subset K$ has FIP, since so has $R\subset S$, and $K$ is a local Artinian ring. We claim that $K$ cannot be a field, because $K'\neq K''$. Otherwise, if $K$ is a field and $K'\neq K''$, then $K'\subset K$, and $T'\subset T$; so that, $M=(T':T)\in\mathrm{Max}(T)$, which gives that $T'\subset T$ is t-closed, a contradiction because $T'\subset S$ is subintegral. 
Using \cite[Theorem 3.8]{ADM}, we have to look at two cases. Assume first that $|k|=\infty$. Then, $K =k[\alpha]$, for some $\alpha\in K$ which satisfies $\alpha^3=0$. In this case, $\dim_k(K)\leq 3$ and is a prime number. But $\dim_k(K)=\dim_{K'}(K)[K':k]$ gives that either $\dim_{K'}(K)=1$ or $[K':k]=1 $. This last condition cannot happen since $k\neq K'$; so that, $\dim_{K'}(K)=1$, which gives $K=K'=K''$, a contradiction. At last, if $k$ is a finite field, so are $K'$ and $K''$. Moreover, they are isomorphic as $k$-algebras and $k\subset K',K''$ are Galois. According to \cite [Proposition 7.8]{DPPS}, we get  a contradiction since $T'\subset S$ is subintegral. Then, $T'=T''$.

Now  $T'=S^+_R$ by Proposition \ref{A4.0} when $S^+_R$ exists.

(2) We next  consider $\varphi:[R,S]\to[R,{}_S^+R]\times[{}_S^+R,S]$ defined by $\varphi(T)=(T\cap{}_S^+R,{}_S^+RT)$ for $T\in[R,S]$. Set $R':=T\cap{}_S^+R$ and $S':={}_S^+RT$, then ${}_S^+R={}_{S'}^+R'$. Let $T'\in[R,S]$ be such that $\varphi(T)=\varphi(T')$.
 Then $T$ and $T'$ are both complements of ${}_{S'}^+R'$ in $[R',S']$. It follows that $T=T'$, applying (1) to the extension $R'\subseteq S'$; so that, $\varphi$ is injective. So we get $|[R,S]|\leq|[R,{}_S^+R]||[{}_S^+R,S]|$ and $|[R,{}_S^+R]|+|[{}_S^+R,S]|-1\leq|[R,S]|$ comes from $[R,{}_S^+R]\cup[{}_S^+R,S]\subseteq[R,S]$ with $[R,{}_S^+R]\cap[{}_S^+R,S]=\{{}_S^+R\}$.
\end{proof}

\begin{corollary}\label{B5.2} Let $R\subset S$ be an integral u-closed FIP extension. Assume that $(R,M)$ is a local ring and $R\neq{}_S^+R\neq S$.  

If $R'\subseteq S''$ is t-closed for any $(R',S')\in[R,{}_S^+R[\times]{}_S^+R,S]$ and for any $S''\in[R',S']$ such that $M'S''\in\mathrm{Max}(S'')$, where $\mathrm{Max}(R')=\{M'\}$, then $|[R,S]|=|[R,{}_S^+R]||[{}_S^+R,S]|$.
\end{corollary}

\begin{proof} First observe that $(R',S')\in[R,{}_S^+R]\times[{}_S^+R,S]$ implies that ${}_{S'}^+R'={}_S^+R$. 
Since $R\subset S$ is an i-extension, any element of $[R,S]$ is a local ring. Using the notation of the proof of Proposition~\ref{B5.1}, we have $|[R,S]|=|[R,{}_S^+R]||[{}_S^+R,S]|$ if and only if $\varphi$ is surjective, if and only if for any $(R',S')\in[R,{}_S^+R]\times[{}_S^+R,S]$, there exists $T\in[R,S]$ such that $R'=T\cap{}_{S}^+R$ and $S'={}_{S}^+RT\ (*)$. If $R'={}_S^+R$, it is enough to take $T=S'$, and if $S'={}_S^+R$, it is enough to take $T=R'$. Then $(*)$ holds also.

Assume that for any $(R',S')\in[R,{}_S^+R[\times]{}_S^+R,S]$ and for any $S''\in[R',S']$ such that $M'S''\in\mathrm{Max}(S'')$, where $\mathrm{Max}(R')=\{M'\}$, this  implies  that $R'\subseteq S''$ is t-closed. 
 We have $R'\neq{}_{S'}^+R'={}_{S'}^tR'\neq S'$. Applying Lemma~\ref{A5.0} to the extension $R'\subseteq S'$, we get that there exists $T\in[R',S']\subseteq [R,S]$ such that $R'=T\cap{}_{S}^+R$ and $S'={}_{S}^+RT$ and $(*)$ holds. It follows that $|[R,S]|=|[R,{}_S^+R]||[{}_S^+R,S]|$.
\end{proof}

\begin{proposition}\label{B5.3} Let $R\subset S$ be a seminormal integral FCP extension such that $(R,M)$ is a local ring and ${}_S^tR\neq R,S$. The following statements hold:
\begin{enumerate}
\item $S$ is a semilocal ring whose maximal ideals are $N_1,\ldots,N_p$. Then, $(R:S)=M=\cap_{i\in\mathbb N_p}N_i$. 
\item ${}_S^tR$ has a complement in $[R,S]$ if and only if $S/N_i\cong S/N_j$, as $R/M$-algebras, for each $i,j\in\mathbb N_p$. 
\item If, in addition, this complement is unique, then it is the co-infra-integral closure of $R\subset S$. 
\end{enumerate}
 \end{proposition}

\begin{proof} (1) $S$ is a semilocal ring because $R\subset S$ is a finite  extension. Let $N_1,\ldots,N_p$ be its maximal ideals. By Proposition \ref{1.91}, $(R:S)$ is an intersection of finitely many maximal ideals of $S$ and of $R$; so that, $(R:S)=M=\cap_{i\in\mathbb N_p}N_i$.

(2) Since $(R:S)=M$ is an ideal of $S$, let $\rho:S\to S/M$ be the canonical map. As $M=\cap_{i\in\mathbb N_p}N_i$, we get that $S':=S/M\cong\prod_{i\in\mathbb N_p}S/N_i$. Set $k:=R/M$ and $K_i:=S/N_i$ for each $i\in\mathbb N_p$, which are field extensions of $k$. Moreover, ${}_{S'}^tk=({}_{S}^tR)/M$ by \cite[Proposition 2.10]{Pic 1}.

Let $T\in[R,S]$; so that, $T':=T/M\in[k,\prod_{i\in\mathbb N_p}K_i]$, up to a $k$-algebra isomorphism. It follows that $T$ is a complement of ${}_S^tR$ in $[R,S]$ if and only if $T'$ is a complement of ${}_{S'}^tk$ in $[k,S']$ if and only if $T'\cap{}_{S'}^tk=k$ and $T'{}_{S'}^tk=S'$. Assume that this condition holds. Then, $k\subset T'$ is t-closed; so that, $T'$ is a field extension of $k$ by \cite[Lemme 3.10]{Pic 1} and $T'\subset S'$ is seminormal infra-integral, giving that $K_i\cong T'$, as $k$-algebras, for each $i\in\mathbb N_p$; so that, $S/N_i\cong S/N_j$, as $R/M$-algebras, for each $i,j\in\mathbb N_p$.

Conversely, assume that $S/N_i\cong S/N_j$, as $R/M$-algebras, for each $i,j\in\mathbb N_p$; so that, $K_i\cong K_j$, as $k$-algebras, for each $i,j\in\mathbb N_p$. Set, for example, $K:=K_1$; so that, $S'\cong K^p$. Let $T'\in[k,S']$ be such that $T'\cong K$ through the diagonal map; so that, $k\subset T'$ is t-closed and $T'\subset S'$ is seminormal infra-integral. There exists $T\in[R,S]$ such that $\rho(T)=T'$ and satisfying $R\subset T$ is t-closed and $T\subset S$ is seminormal infra-integral. Then, $T\cap{}_{S}^tR=R$ and $T{}_{S}^tR=S$, which shows that $T$ is a complement of ${}_{S}^tR$.

(3) Assume, in addition, that the complement of ${}_{S}^tR$ is unique. We are going to show that it is the co-infra-integral closure of $R\subset S$. We keep the previous notation with $T$ the complement of ${}_{S}^tR$, and, to make simpler the proof, assume that $S'=K^p$ (which is satisfied up to an isomorphism). Let $U\in[R,S]$ be such that $U\subseteq S$ is infra-integral. Then, $U$ is a semilocal ring. Set $\mathrm{Max}(U):= 
\{M_1,\ldots,M_n\}$; so that, $M=\cap_{j\in\mathbb N_n}M_j$. Since $U\subset S$ is infra-integral, we have $U/M_j\cong S/N_i$ for any $N_i$ lying over $M_j$. But $S/N_i\cong K$, as $R/M$-algebra, for each $i\in\mathbb N_p$ leads to $U/M_j\cong K$, as $R/M$-algebra, for each $j\in\mathbb N_n$. Then, $U/M\cong K^n$. In particular, there exists $V\in[R,U]$ such that $\rho(V)\cong K$. This shows that $R\subseteq V$ is t-closed and $V\subseteq S$ is seminormal infra-integral. This means that $V$ is a complement of ${}_{S}^tR$. By uniqueness of this complement, it follows that $V=T$, giving $T\subseteq U$. Then, $T$ is the  co-infra-integral closure of $R\subset S$. 
\end{proof}

\begin{remark}\label{B5.4} Let $R\subset S$ be an FIP seminormal extension such that ${}_S^tR\neq R,S$. Assume that $(R,M)$ is a local ring. Then, $S$ is semilocal since $R\subset S$ is a finite extension. Let $N_1,\ldots,N_p$ be the maximal ideals of $S$. As in Corollary~\ref{B5.2}, we can define a map $\varphi:[R,S]\to[R,{}_S^t R]\times
[{}_S^tR,S]$ by $\varphi(T):=(T\cap{}_S^tR,{}_S^tRT)$. But, contrary to the results of this corollary, we are going to show that $\varphi$ is not always injective and never surjective.  

Assume first that there exists $T$ such that $R\subset T$ is t-closed and $T\subset S$ is infra-integral; so that, $\varphi(T):=(R,S)$ and $T$ is a complement of ${}_S^tR$. Then $M\in\mathrm{Max}(T)$. Each $N _i$ lies over $M$ in $T$ and $M=\cap_{i=1}^pN_i=(R:S)$ by Proposition~\ref{1.91}. Moreover, $S/N_i\cong T/M$ for each $i=1,\ldots,p$. Set $k:=R/M,\ K:=T/M$ and $L:=S/M\cong\prod_{i=1}^pS/N_i\cong K^p$. Assume that the Galois group of the field extension $k\subset K$ has more than one element and let $\sigma$ be such an element different from the identity. Since $R\subset T$ has FIP, so has $k\subset K$, and there exists $x\in K$, algebraic over $k$, such that $K=k[x]$ in view of the Primitive Element Theorem. In particular, $\sigma(x)\neq x$. Set $y:= (\sigma(x),x,\ldots,x)\in K^p$ and $K':=k[y]$. Then $K$ and $K'$ are $k$-isomorphic distinct $k$-subalgebras of $K^p$, and the inverse image $T'$ of $K'$ under the canonical map $S\to S/M$ is an $R$-subalgebra of $S$ such that $R=T'\cap{}_{S}^tR$ and $S={}_S^tRT'$, with $T\neq T'$. It follows that $\varphi(T)=\varphi(T')$ and $\varphi$ is not injective. See also Example \ref{1.434}, where $\varphi(K)=\varphi(L)$ with $K\neq L$. But, if $k\subset K$ is a radicial extension, its Galois group has only one element, and $K^p$ has no $k$-subalgebra isomorphic to $K$ distinct from $K$. Then there is a unique complement of ${}_S^tR$. This reasoning still holds for any $(R',S')\in[R,{}_S^tR]\times[{}_S^tR,S]$. In this case, $\varphi$ is injective.

Now, we claim that $\varphi$ is never surjective. For each $i\in\mathbb N_p$, set $M_i:=N_i\cap{}_S^tR$. Then ${}_S^tR/M_i\cong R/M$. Let $R_1\in[{}_S^tR,S]$ be such that ${}_S^tR\subset R_1$ is minimal (inert) with conductor $M_1\in\mathrm{Max}({}_S^tR)$ (after a suitable reordering), and set $N'_j:=N_j\cap R_1$ for $j\in\{2,\ldots,p\}$. It follows that ${}_{R_1}^tR={}_S^tR$ in view of Proposition~\ref{1.31} and $R\subset R_1$ is an FIP seminormal extension. Moreover, $R_1/N'_j\cong {}_S^tR/M_j\cong R/M$ for each $j\in\{2,\ldots,p\}$. But $R_1/M_1\not\cong{}_S^tR/M_1$; so that, $R_1/M_1\not\cong R_1/N'_j$ for each $j\in\{2,\ldots,p\}$. It follows by Proposition~\ref{B5.3} that ${}_{R_1}^t R$  does not admits a complement in $[R,R_1]$; so that, there does not exist $T$ such that ${}_S^tRT=R_1$ and $R={}_S^tR\cap T$. Then, $\varphi$ is not surjective. 
 \end{remark}

\begin{proposition}\label{B5.5} Let $R\subset S$ be an  infra-integral  FCP extension. The map $\varphi:[R,S]\to[R,{}_S^+R]\times[{}_S^+R,S]$ defined by $\varphi(T):=(T\cap{}_S^+R,{}_S^+RT)$ is injective. 

In particular, if $R\subset S$ has FIP, then $|[R,S]|\leq|[R,{}_S^+R]||[{}_S^+R,S]|$.
 \end{proposition}

\begin{proof} Let $U,U'\in[R,S]$ be such that $\varphi(U)=\varphi(U')$. Set $R':=U\cap{}_S^+R$ and $S':={}_S^+RU$. Then, ${}_S^+R={}_{S'}^+R'$, with $R'\subseteq S'$ infra-integral. Moreover, $U$ and $U'$ are both complements of ${}_{S'}^+R'$. As Theorem~\ref{1.423} asserts that $R'\subseteq S'$ admits a co-subintegral closure $T$, it follows from  Proposition \ref{A4.0} that $T=U=U'$; so that, $\varphi$ is injective. 
 Then,   $|[R,S]|\leq|[R,{}_S^+R]||[{}_S^+ R, S]|$ if $R\subset S$ has FIP.
\end{proof}

The following proposition gives us another way to calculate $|[R,S]|$, when $R\subseteq S$ is an integral FIP extension with $R$ local.

\begin{proposition}\label{B5.7} Let $R\subset S$ be an FIP integral extension, over the local ring $(R,M)$ and set $\mathrm{Max}(S):=\{N_1,\ldots,N_p\}$. Let $n[R,S]$ be the number of complements of ${}_S^tR$ in $[R,S]$. The following conditions hold:

\begin{enumerate}
\item There exists a greatest $R$-subalgebra $R'$ of $S$ such that $M$ is an ideal of $R'$.  

\item If $S/N_i\not\cong S/N_j$ as $R/M$-algebras, for some $i\neq j\in\mathbb N_p$, then $n[R,S]=0$.

\item Assume that $n[R,S]\neq 0$. Then, any complement $T$ of ${}_S^tR$ is in $[R,R']$ and satisfies $S/N_i\cong T/M$ for any $i\in\mathbb N_p$. Let $P(X)\in(R/M)[X]$ be the minimal monic polynomial of some $y\in T/M$ such that $T/M=(R/M)[y]$. For any $i\in\mathbb N_p$, there exists some $y_i\in S/N_i$ such that $S/N_i=(R/M)[y_i]$ where $P(X)$ is also the minimal monic  polynomial of  $y_i$.
 Moreover, $n[R,S]=|\{(R/M)[x]\mid x\in R'/M,\ P(x)=0\}|$. 
\end{enumerate}
\end{proposition}

\begin{proof} First observe that $S$ is semilocal by Propositions~\ref{1.31} and ~\ref{1.91}.

(1) Set $\mathcal{R}:=\{T\in[R,S]\mid M$ is an ideal of $T\}$. Then $\mathcal{R}$ is a finite nonempty set. It has maximal elements. Assume that $R'$ and $R''$ are two maximal elements of $\mathcal{R}$. If $R'\neq R''$, then $R'\subset R'R''$ with $R'R''\in\mathcal{R}$, a contradiction. Then $\mathcal{R}$ has a unique maximal element, which is its greatest element. Let $R'$ be this element. 

(2) Assume that $S/N_i\not\cong S/N_j$ as $R/M$-algebras, for some $i\neq j\in\mathbb N_p$ and that $n[R,S]\neq 0$. Then ${}_S^tR$ admits some complement $T$ such that $R\subseteq T$ is t-closed; so that, $T$ is local with maximal ideal $M$ and $T\subseteq S$ is infra-integral. 
It follows that $S/N_i\cong T/M$ for each $i\in\mathbb N_p$. In particular, $S/N_i\cong S/N_j$ as $R/M$-algebras, for each $i\neq j\in\mathbb N_p$, a contradiction, whence  $n[R,S,]=0$.

(3) According to (2), $n[R,S]\neq 0$ implies that $S/N_i\cong S/N_j$ as $R/M$-algebras, for each $(i,j)\in\mathbb N_p$. 
 Assume that ${}_S^tR$ admits some complement $T$; so that,
    $R\subseteq T$ is t-closed. Therefore, $T$ is local with maximal ideal $M$, whence  $M=(R:T)$, giving $T\in[R,R']$. Then $R/M\subseteq T/M$ is a finite FIP field extension; so that, there exists some $y\in T/M$ such that $T/M=(R/M)[y]$. Since $T\subseteq S$ is infra-integral; so that, $S/N_i\cong T/M$ for each $i\in\mathbb N_p$, let $\varphi_i:T/M\to S/N_i$ be the $R/M$-algebra isomorphism, and set $y_i:=\varphi_i(y)$ for each $i\in\mathbb N_p$. It follows that $S/N_i=\varphi_i(T/M)=\varphi_i(R/M[y])=R/M[\varphi_i(y)]=R/M[y_i]$. Let $P(X)\in(R/M)[X]$ be the minimal monic polynomial of $y$, then $P(X)$ is also the minimal monic polynomial of  $y_i$ and $T/M\cong (R/M)[X]/(P(X))\cong S/N_i$ for each $i\in\mathbb N_p$.
     
Let $T'$ be another complement of ${}_S^tR$. As we have just seen before, $T'\in[R,R']$ and is local with maximal ideal $M$ and $T'\subseteq S$ is infra-integral, which implies that $S/N_i\cong T'/M$ for each $i\in\mathbb N_p$. Then, there exists an $R/M$-algebra isomorphism $\theta:T\to T'$. Setting $x:=\theta(y)$, we get that $P(X)$ is also the minimal monic polynomial of  $x$.

Conversely, let $x'\in R'/M$ be such that $P(x')=0$ and let $T'\in[R,R']$ be the inverse image of $(R/M)[x']$ by the canonical map $\rho:R'\to R'/M$. Then $T'$ is an $R$-subalgebra of $S$ with maximal ideal $M$, because $T'/M=(R/M)[x']\cong(R/M)[X]/(P(X))$, which is a field; so that, $R\subseteq T'$ is t-closed, since $M=(R:T')$. As we have $T'/M\cong(R/M)[X]/(P(X))\cong S/N_i$ for each $i\in\mathbb N_p$, it follows that $T'\subseteq S$ is infra-integral. Then $T'$ is a complement of ${}_S^tR$. To end, let $x,x'\in R'/M$ be such that $P(x)=P(x')=0$. If $(R/M)[x]=(R/M)[x']$, their inverse images by $\rho$ give the same $T$. Then $n[R,S]=|\{(R/M)[x]\mid x\in R'/M,\ P(x)=0\}|$. 
\end{proof}

We deduce the following theorem.

\begin{theorem}\label{B5.8} Let $R\subset S$ be an FIP integral extension, over the local ring $(R,M)$. For each $(R',S')\in[R,{}_S^tR]\times[{}_S^tR,S]$, let $n[R',S']$ be the number of complements of ${}_{S'}^tR'$ in $[R',S']$. Then $|[R,S]|=\sum_{(R',S')\in[R,{}_S^tR]\times[{}_S^tR,S]}n[R',S']$.
 \end{theorem}

\begin{proof} Let $(R',S')\in[R,{}_S^tR]\times[{}_S^tR,S]$. As we already have seen, we have ${}_S^tR={}_{S'}^tR'$.

Let $T\in[R,S]$. Set $R':=T\cap{}_S^tR$ and $S':={}_S^tRT$; so that, $(R',S')\in[R,{}_S^tR]\times[{}_S^tR,S]$. Then, $T\in[R',S']$ and is a complement of ${}_{S'}^tR'={}_S^tR$ in $[R',S']$. 
Moreover, for any $(R'',S'')\in[R,{}_S^tR]\times[{}_S^tR,S]$ such that $(R'',S'')\neq(R',S')$, we get that $T$ is not a complement of ${}_{S''}^tR''={}_S^tR$ in $[R'',S'']$. 
Consider the map $\varphi:[R,S]\to[R,{}_S^tR]\times[{}_S^tR,S]$ defined by $\varphi(T):=(T\cap{}_S^tR,{}_S^tRT)$. Since $n[R',S']=|\{T\in[R',S']\mid R'=T\cap{}_{S'}^tR',\ S'={}_{S'}^tR'T\}|$, we get that $n[R',S']=|\{T\in [R,S]\mid\varphi(T)=(R',S')\}|$. It follows that $|[R,S]|=\sum_{(R',S')\in[R,
{}_S^tR]\times[{}_S^tR,S]}n[R',S']$, because $\varphi(T)=(R',S')$ implies that $T\in[R',S']$. We may remark that it can be that $n[R',S']=0 $ for some $(R',S')\in[R,{}_S^tR]\times[{}_S^tR,S]$ according to  Proposition ~\ref{B5.7}(2).
\end{proof}

\begin{remark}\label{B5.9} In the preceding theorem, we may define $n[R',S']$ even if $R'$ is not a local ring. But, since $R'\subseteq S'$ is an integral FIP extension, it is a $\mathcal B$-extension. 
Moreover, for any $M\in\mathrm{MSupp}_{R'}(S'/R')$, we have ${}_{S'_M}^tR'_M=({}_{S'}^tR')_M$; so that, for some $T\in[R,S]$, we have 
$R'=T\cap{}_{S'}^tR'$ and $S'={}_{S'}^tR'T$ if and only if $R'_M=T_M\cap {}_{S'_M}^tR'_M$ and $S'_M={}_{S'_M}^tR'_MT_M$ for any $M\in\mathrm{MSupp}_{R'}(S'/R')$. It follows that $n[R',S']=\prod_{i=1}^n n[R'_{M_i},S'_{M_i}]$, which we can calculate using Proposition 
~\ref{B5.7}.  
\end{remark}

We now consider  some types of integral FIP extensions. We still work with an extension $R\subset S$ such that $(R,M)$ is a local ring since $R\subset S$ is a $\mathcal B$-extension. In the following Proposition, we recall and generalize a result of \cite{DPP3}, using the Nagata rings $R(X)$ and $S(X)$.

\begin{proposition}\label{B5.10} Let $R\subset S$ be an FIP subintegral extension. Assume that $(R,M)$ is a local ring.
\begin{enumerate}
\item \cite[Proposition 4.13]{DPP3} If $|R/M|=\infty$, then $R\subset S$ is chained and $|[R,S]|=\ell[R,S]+1$.

\item Assume that $|R/M|<\infty$ and let $n$ be the nilpotency index of $M/(R:S)$. If $R+SM^2\subseteq S$ is chained and ${\mathrm L}_R(MS/M)=n-1$, then $R\subset S$ is chained and $|[R,S]|=\ell[R,S]+1={\mathrm L}_R(N/M)$, where $N$ is the maximal ideal of $S$.
\end{enumerate}
\end{proposition}

\begin{proof} Since $R$ is a local ring and $R\subset S$ is an FIP subintegral extension, $S$ is also a local ring.

(2) Assume that $|R/M|<\infty,\ R+SM^2\subseteq S$ is chained and ${\mathrm L}_R(MS/M)=n-1$. By \cite[Corollary 3.31]{DPP3}, we get that $R(X)\subseteq S(X)$ has FIP. Moreover, $(R(X),MR(X))$ is a local ring such that $|R(X)/MR(X)|=\infty$, and $R(X)\subset S(X)$ is a finite subintegral extension \cite[Lemma 3.15]{DPP3}. Then, (1) gives that $ R(X)\subset S(X)$ is chained. But the map $\varphi:[R,S]\to[R(X),S(X)]$, defined by $\varphi(T):=T(X)$ for each $T\in[R,S]$, is an order-preserving and order-reflecting injection \cite[Lemma 3.1(d)]{DPP3}; so that, $R\subset S$ is also chained, giving $|[R,S]|=\ell[R,S]+1={\mathrm L}_R(N/M)+1$ by \cite[Lemma 5.4]{DPP2}.
\end{proof} 

\begin{remark}\label{B5.101} (1) When $(R,M)$ is a local ring with $|R/M|<\infty$ such that $R(X)\subseteq S(X)$ has not FIP, the previous reasoning does not hold, as well as the equalities of Proposition~\ref{B5.10} (2). We use \cite[Example 3.12]{DPP3} to illustrate this situation. Let $K$ be a finite field with $n:=|K|$. Set $T:=K[Y]/(Y^4)=K[y]$, where $y$ is the class of $Y$ in $T$; so that, $y^4=0$. Then, $T$ is a finite local Artinian ring with maximal ideal $N:=yT$. Moreover, $K\subset T$ is finite subintegral. Let $S\in[K,T]$ be such that $S\subset T$ is minimal (ramified). Then, $S$ is also a finite local Artinian ring. Let $M$ be its maximal ideal. Then, $N^2\subseteq M$ gives $y^2,y^3\in M$; so that, $K [y^2,y^3] \subseteq S\subset T$ implies $K[y^2,y^3]=S$ by considering the dimensions of the different $K$-vector spaces. Obviously, $K\subset K':=K[y^3]$ and $K\subset K_{\alpha}:=K[y^2+\alpha y^3]$ are minimal extensions, all distinct, for any $\alpha\in K$. Moreover, $K'\subset S$ and $K_{\alpha}\subset S$ are also minimal extensions. Then, $[K,T]=\{K,K',K_{\alpha},S,T\mid \alpha\in K\}$; so that, $|[K,T]|=n+4$.  
 As $K\subset T$ is not chained, $K(X)\subset T(X)$ has not FIP, according to \cite[Corollary 3.31]{DPP3}.

(2) We can add another property about $|[R,S]|$ for an FIP subintegral extension $R\subseteq S$ when $(R,M)$ is a finite local ring. Let $N$ be the maximal ideal of $S$ and let $T\in[R,S]$ be such that $T\subset S$ is a minimal (ramified) extension. Then, $P:=(T:S)$ is the maximal ideal of $T$ and an $N$-primary ideal of $ S$ such that $P$ and $N$ are adjacent in $S$ (equivalently, $\dim_{R/M}(N/P)=1$ by Theorem~\ref{minimal}). Assume that there exists $T'\in[R,S]$ such that $T'\subset S$ is minimal with $(T':S):=P$. Then, $T=T'=R+P$. Conversely, if $P$ is an $N$-primary ideal of $S$ such that $P$ and $N$ are adjacent in $S$ and $M=P\cap R$, setting $S':=R+P$, we get that $S'\subset S$ is a minimal ramified extension. Then, the map $\psi:\{P$ an $N$-primary ideal of $S\mid P$ and $N$ adjacent, $M=P\cap R\}\to\{S'\in[R,S]\mid S'\subset S$ minimal$\}$ defined by $\psi(P):=R+P$ is a bijection. Let $m$ be the number of $N$-primary ideals $P$ of $S$ such that $P$ and $N$ are adjacent and $M=P\cap R$. In particular, if ${\mathrm L}_R(N/M)=2$, then $\ell[R,S]=2$ by \cite[Lemma 5.4]{DPP2} and $|[R,S]|=m+2$. It follows by \cite[Theorem 6.1 (6)]{Pic 6} that either $|[R,S]|=3\ (*)$ or $|[R,S]|=|R/M|+3\ (**)$. In case $(*)$, we get that $R\subseteq S$ is chained and $m=1$. In case $(**)$, we get that $|R/M|=m-1$.  
\end{remark}

\begin{proposition}\label{B5.103} Let $R\subset S$ be an FIP integral t-closed extension with $C:=(R:S)$. Then $|[R,S]|=\prod_{M\in\mathrm{V}_R(C)}|[S/MS,R/M]|$. In fact, $|[S/MS,R/M]|$ is the number of $R/M$-subalgebras of $S/SM$.
 \end{proposition}

\begin{proof} Since $R\subset S$ is t-closed, $C=(R:S)$ is an intersection of maximal ideals of $R$ and $S$ by Proposition  ~\ref{1.91}. Since $R\subset S$ is an i-extension, for each $M\in\mathrm{V}_R(C)$, there is one and only one maximal ideal $N$ of $S$ lying over $M$. We claim that $N$ is of the form $N=SM$. First, $SM\subseteq N$. But $MR_M=MS_M=MS_N=NS_N$ by \cite[Proposition 2, page 40]{Bki A1} and Proposition~\ref{1.31}, since $S_M=S_N$. Moreover, for $M'\in\mathrm{Max}(R),\ M'\neq M$, we have $MS_{M'}=S_{M'}=NS_{M'}$. Then, $N=SM$. Since $R\subset S$ is a $\mathcal B$-extension, we have $|[R, S]|=\prod_{M\in\mathrm{MSupp}(S/R)}|[R_M,S_M]|$. But, $MR_M=(R_M:S_M)$; so that, $|[R_M/(MR_ M), S_M/(MR_M)]|=|[R_M,S_M]|$ by \cite[Proposition 3.7]{DPP2}. Since $ MR_M\in\mathrm{Max}(S_M)$, it follows that $R_M/(MR_M)\to S_M/(MR_M)$ is a field extension. At last, $R_M/(MR_M)\cong R/M$ and $S_M/(MR_M)\cong S/MS$.
\end{proof}

We end this section with open questions. In Lemma ~\ref{A5.0}, we proved that for a u-closed integral extension $R\subset S$ where $R$ is a local ring, under some special assumptions, if ${}_{S}^+R$ has a unique complement in $[R,S]$, then it is the co-subintegral closure of $R\subset S$. We can ask if the same result holds in the more general situation of Proposition ~\ref{B5.1}. In a similar way, Proposition ~\ref{B5.3} shows that for a seminormal integral extension $R\subset S$ where $R$ is a local ring, under some special assumptions, if ${}_{S}^tR$ has a unique complement in $[R,S]$, then it is the co-infra-integral closure of $R\subset S$. We can ask if the same result holds in the more general situation of Proposition~\ref{B5.7}.

\end{document}